\documentclass[11pt,a4paper,reqno]{amsart}

\usepackage[utf8]{inputenc}
\usepackage[T1]{fontenc}
\usepackage{microtype}
\usepackage{tikz-cd}
\usepackage{booktabs}
\usepackage{enumitem}

\usepackage[margin=1.4in]{geometry}

\usepackage{mathtools}
\usepackage{amsthm}
\usepackage{amssymb}
\usepackage{mathrsfs}

\usepackage[sorting=nyt,maxnames=4]{biblatex}
\usepackage{hyperref}
\hypersetup{
    colorlinks,
    linkcolor={red!50!black},
    citecolor={blue!70!black},
    urlcolor={blue!70!black}
}
\usepackage[capitalise]{cleveref}

\newtheorem{theorem}[equation]{Theorem}
\newtheorem*{theorem*}{Theorem}
\newtheorem{conjecture}[equation]{Conjecture}
\newtheorem{proposition}[equation]{Proposition}
\newtheorem{lemma}[equation]{Lemma}
\newtheorem*{corollary*}{Corollary}
\newtheorem{corollary}[equation]{Corollary}

\newtheorem{introtheorem}{Theorem}

\newtheorem{introcorollary}[introtheorem]{Corollary}

\theoremstyle{definition}
\newtheorem{definition}[equation]{Definition}
\newtheorem{notation}[equation]{Notation}
\newtheorem{defprop}[equation]{Definition / Proposition}

\theoremstyle{remark}
\newtheorem{remark}[equation]{Remark}
\newtheorem{example}[equation]{Example}

\numberwithin{equation}{section}
\numberwithin{figure}{section}
\crefname{section}{\S\kern -3pt}{\S\S}
\crefname{equation}{}{}

\let\oldtocsection=\tocsection
\let\oldtocsubsection=\tocsubsection

\renewcommand{\tocsection}[3]{\hspace{0em}\oldtocsection{#1}{#2}{#3}}
\renewcommand{\tocsubsection}[3]{\hspace{2.5em}{\oldtocsubsection{#1}{#2}{#3}}}

\DeclareMathOperator{\coker}{coker}
\DeclareMathOperator{\Spec}{Spec}
\DeclareMathOperator{\Sym}{Sym}
\DeclareMathOperator{\Tot}{Tot}
\DeclareMathOperator{\Hom}{Hom}

\DeclareMathOperator{\shHom}{\mathcal{H}\kern -1pt \textit{om}}
\DeclareMathOperator{\RshHom}{\dR\mathcal{H}\kern -1pt \textit{om}}
\DeclareMathOperator{\shExt}{\mathcal{E}\kern -1pt \textit{xt}}

\DeclareMathOperator{\Crit}{Crit}
\DeclareMathOperator{\RHom}{\dR Hom}
\DeclareMathOperator{\RG}{\dR\Gamma}
\DeclareMathOperator{\Perf}{\mathfrak{Perf}}
\DeclareMathOperator{\Coh}{Coh}
\DeclareMathOperator{\Bl}{Bl}

\DeclareMathOperator{\QCoh}{QCoh}
\DeclareMathOperator{\MF}{MF}
\DeclareMathOperator{\EF}{EF}

\DeclareMathOperator{\Sing}{Sing}

\newcommand{\id}{\mathrm{id}}

\newcommand{\pt}{\mathrm{pt}}
\newcommand{\dR}{\mathrm{R}} 

\newcommand{\qc}{\mathrm{qc}}
\newcommand{\norm}{\mathrm{norm}}

\newcommand{\eps}{\epsilon}
\newcommand{\tilsigma}{{\tilde\sigma}}
\newcommand{\tilrho}{{\tilde\rho}}

\newcommand{\tilX}{{\widetilde X}}
\newcommand{\tilL}{{\widetilde L}}
\newcommand{\Z}{\mathbb{Z}}
\newcommand{\C}{\mathbb{C}}
\newcommand{\A}{\mathbb{A}}
\newcommand{\calA}{\mathcal{A}}

\newcommand{\calE}{\mathcal{E}}
\newcommand{\scrC}{\mathscr{C}}
\newcommand{\calF}{\mathcal{F}}

\renewcommand{\P}{\mathbb{P}}

\renewcommand{\O}{\mathcal{O}}
\newcommand{\W}{\mathcal{W}}

\title{Odd Kn\"orrer periodicity as a double cover}
\author{Calum Crossley}
\date{}

\begin{document}

\begin{abstract}
    We prove that the derived category of a branched double cover is equivalent
    to a category of matrix factorizations for a fiberwise quadratic potential
    on the associated line bundle. This requires the linear fiber coordinate to
    have odd cohomological degree, so we work with matrix factorizations which
    no longer have the traditional splitting into even and odd parts.
\end{abstract}

\maketitle
\tableofcontents

\section{Introduction}

Kn\"orrer periodicity was originally motivated by a desire to study how
stabilizations
\begin{equation*}
    w\rightsquigarrow w+x_1^2+\cdots+x_r^2
\end{equation*}
affect the category of matrix factorizations \cite{Kn}. There are two
equivalences
\begin{equation*}
    \MF(\A^2,x^2+y^2) \simeq D^b(\pt) \qquad \text{and} \qquad
    \MF([\A^1/\Z_2],x^2) \simeq D^b(\pt),
\end{equation*}
which show by a categorical Kunneth formula that even rank stabilizations
leave the category unchanged, while odd rank stabilizations (which are
non-trivial \cref{eqn:cliff}) can be undone by working equivariantly with
respect to the involution $x\mapsto-x$.

Given such pointwise equivalences, it is natural to vary the input data in a
family to obtain global statements. For example, there is a family of even rank
quadratic forms given by the evaluation pairing $q:V\times V^\vee\to\O_B$ for a
vector bundle $V\to B$, and one global version of even Kn\"orrer periodicity is
as follows:
\begin{equation*}
    \MF(V\times V^\vee, q) \simeq D^b(B).
\end{equation*}
Replacing $B\subset V$ by a more general regular embedding leads to a famous
equivalence exhibiting the derived category of a complete intersection as a
category of matrix factorizations over the total space of a vector bundle
(\cite{O1}, \cite{I}, \cite{Sh}, \cite{H}, etc.).

The purpose of this paper is to give a global statement modelled on a pointwise
equivalence for rank one quadratic forms, which involves two points:
\begin{equation} \label{eqn:odd}
    \MF_{|x|=1}(\A^1,x^2) \simeq D^b(\text{2 pts}).
\end{equation}
Varying this in a family (with potential degenerations) leads to an equivalence
involving a branched double cover. Many of our motivating examples are stacky;
for example, the function $x^2$ on $[\A^1/\Z_2]$ can be seen as a family of
rank one quadratic forms over $B\Z_2$, and the equivalence
$\MF([\A^1/\Z_2],x^2)\simeq D^b(\pt)$ follows from viewing a point as a double
cover of $B\Z_2$.

What makes our discussion novel is that the equivalence \cref{eqn:odd} requires
working with odd grading data. Without grading data, there is a $\Z_2$-graded
equivalence
\begin{equation} \label{eqn:cliff}
    \MF(\A^1,x^2) \simeq D^b_{|\eps|=1}(\C[\eps]/(\eps^2-1)),
\end{equation}
where $\C[\eps]/(\eps^2-1)$ is the Clifford algebra associated to the rank one
quadratic form, which is $\Z_2$-graded with $|\eps|=1$.\footnote{We take this
as a cohomological grading, so the derived category consists of dg-modules
rather than chain complexes of graded modules.} This Clifford algebra does not
split at the graded level, because the factors $\eps\pm1$ are inhomogeneous.
For \cref{eqn:odd} we need to put a non-trivial $\Z_2$-grading on $\C[x]$ where
$|x|=1$, which gives a similar equivalence but with $|\eps|=0$, allowing for a
graded splitting:
\begin{equation*}
    \MF_{|x|=1}(\A^1,x^2)
        \simeq D^b(\C[\eps]/(\eps^2-1))
        \simeq D^b\biggl(\frac{\C[\eps]}{\eps-1}
            \oplus \frac{\C[\eps]}{\eps+1}\biggr)
        \simeq D^b(\text{2 pts}).
\end{equation*}
This non-trivial $\Z_2$-grading lifts to a $\Z$-grading of the curved algebra
$(\C[x],x^2)$, where the trivial one did not (curvature must have degree 2,
forcing $|x|=1$).

A preliminary part of the paper is dedicated to explaining the correct
framework for matrix factorizations with odd grading data, because most
existing approaches to graded matrix factorizations in the literature are not
adequate for our purposes. One reason for this is that many definitions are
based on the splitting of matrix factorizations into pairs of sheaves
$\calE^0\rightleftarrows\calE^1$ interchanged by the differential, which is a
phenomenon unique to evenly-graded rings. Our basic example already involves
objects which cannot be described in this way: the two orthogonal exceptional
objects in $\MF_{|x|=1}(\A^1,x^2)$ arising from \cref{eqn:odd} are the rank one
curved modules
\begin{equation*}
    \begin{tikzcd}
        \C[x] \ar[loop right,distance=1em,"\pm x"]
    \end{tikzcd}
\end{equation*}
equipped with differentials $\pm x\cdot\id_{\C[x]}$, which are indecomposable
as $\C[x]$-modules. Another reason is that odd gradings break
super-commutativity of the structure sheaf, meaning for example that $\RshHom$
is no longer defined in general (see \cref{sec:odd}). This is why even grading
data ($-1\in\C^*_R$ acting trivially) was taken as an axiom in \cite{Seg}.

Our approach is simply to take the object-level definition from \cite{Seg},
which can be applied without the evenness assumption, and follow the by-now
standard method for defining a derived category of matrix factorizations using
acyclic objects, which does not make use of $\RshHom$ \cite{O3}, \cite{P}. We
observe that a version of $\RshHom$ defined over an associated $\Z_2$-quotient
stack still exists when the grading is odd (\cref{sec:loc}), allowing for
local-to-global computations in this more general context.

As some philosophical justification for working with these
non-super-commutative spaces, in \cref{sec:mirr} we note that oddly-graded
matrix factorizations arise naturally in some 1-dimensional examples of mirror
symmetry, corresponding to wrapped Fukaya categories of orbifold surfaces
\cite[\S9]{Sei}, \cite{BSW}.

\subsection{Statements} \label{sec:stmt}
Our main result, putting \cref{eqn:odd} into a family, is the following:

\begin{introtheorem}[\cref{thm:main}] \label[introtheorem]{thm:intro}
    Let $\tilX=\{y^2=f\}\subset\Tot\{L\xrightarrow{\pi}X\}$ be a branched
    double cover of a variety $X$, where $f\in H^0(X,L^2)$ cuts out the branch
    locus and $y\in H^0(\Tot L,\pi^*L)$ is the tautological linear fiber
    coordinate. There is an equivalence
    \begin{equation*}
        \MF(\Tot\{L^{-1}\xrightarrow{\tau}X\}, fq^2) \simeq D^b(\tilX),
    \end{equation*}
    where $L^{-1}$ is graded by a $\C^*$-action acting fiberwise with weight 1,
    so that the tautological linear fiber coordinate $q\in H^0(\Tot
    L^{-1},\tau^*L^{-1})$ is odd.
\end{introtheorem}

Fiberwise over $X$, we have $\A^1$ equipped with a superpotential which is a
multiple of $q^2$. This is equivalent to $\MF_{|q|=1}(\A^1,q^2)\simeq
D^b(\text{2 pts})$ except when $f$ vanishes, where we instead get
$D^b_{|q|=1}(\C[q])\simeq D^b(\C[y]/y^2)$ by Koszul duality. These are the
expected fibers for a branched double cover.

\begin{introcorollary}[\cref{cor:root}] \label[introcorollary]{cor:intro}
    In the same setup, we also have an equivalence
    \begin{equation*}
        \MF([\Tot L^{-1}/\Z_2],fq^2) \simeq D^b(\sqrt{f/X}),
    \end{equation*}
    where $\sqrt{f/X}=[\tilX/\Z_2]$ is the associated root stack and $\Z_2$
    acts fiberwise on $\Tot L^{-1}$.
\end{introcorollary}

This corollary stays within the realm of evenly-graded matrix factorizations,
because $-1\in\C^*$ acts trivially on the quotient stack. It gives a non-linear
version of the rank one case of homological projective duality relating a
linear system of odd rank quadrics to the root stack of the discriminant
divisor, e.g. \cite[\S4.2]{ASS}.

\begin{remark}
    A natural attempt at generalization would be to consider a degree $n$
    cyclic cover $\{y^n=f\}$ associated to $f\in H^0(X,L^n)$, and the analogous
    category $\MF(\Tot L^{-1},fq^n)$. But there is no equivalence here: the
    analogue of the Clifford algebra for $n\ge3$ is $\C[\eps]/\eps^2$ with an
    $n$-ary $A_\infty$ operation \cite{D}, and this has little relation to
    $D^b(\text{$n$ pts})=D^b(\C^{\oplus n})$. The category also cannot be
    $\Z$-graded, since that would require $|q|=\frac{2}{n}$. The orbifold
    $\MF([\Tot L^{-1}/\Z_n],fq^n)$ does have such a grading, but its fiberwise
    behaviour is governed by the category $\MF([\A^1/\Z_n],q^n)$ which is
    equivalent to the $A_{n-1}$ quiver rather than a single point, so
    \cref{cor:intro} also fails to generalize. This orbifold probably relates
    to the results of \cite{KP} instead.
\end{remark}

\subsection{Odd to even}
As is natural in such contexts, the full statement of \cref{thm:intro} allows
for a non-zero superpotential $w$ on $X$ as well as appropriate grading data,
so that the result can be iterated. That is, we have a more general equivalence
\begin{equation} \label{eqn:withw}
     \MF(\Tot L^{-1},w+fq^2) \simeq \MF(\tilX,w).
\end{equation}
If we apply this twice, taking $L$ to be trivial and $f$ to be the constant
function 1,\footnote{This corresponds to a trivial double cover.} then on one
hand we get an even rank stabilization of the superpotential, and on the other
hand we get an iterated double cover. This seems paradoxical at first, since
Kn\"orrer periodicity implies that even rank stabilizations should be trivial,
while an iterated double cover is non-trivial. The explanation for this is
that we need to assume $X$ is evenly-graded for the naive generalization to
hold (see \cref{sec:tech}), and this is typically false when iterating the
result because $q$ is odd. The correct statement when $X$ is oddly-graded
involves certain non-commutative algebras, which become the Clifford algebras
for even rank quadratic forms when iterating an even number of times (see
\cref{ex:kp}). These are Morita trivial despite doubling in size, which
resolves the paradox.

The double cover $\tilX$ also has an equivalence coming from global even
Kn\"orrer periodicity \cite{H} applied to its construction as a hypersurface in
a line bundle. Our result is a square root of this equivalence, using the fact
that $L^{-1}$ is a double cover of $L^{-2}$ ramified over the zero section:
\begin{equation} \label{eqn:sqrt}
    \begin{tikzcd}
        X &[4em]
        (L^{-2},fp) &[4em] \\
        \tilX \ar[u,"2:1","y^2=f"']
            \ar[rr,bend right,distance=4em,"\text{even Kn\"orrer}","\simeq"']
            \ar[r,"\text{\cref{eqn:withw}, $w=0$}","\simeq"'] &
        (L^{-1},fq^2) \ar[u,"2:1","q^2=p"']
            \ar[r,"\text{\cref{eqn:withw}, $w=fp$}","\simeq"'] &
        (L\oplus L^{-2},(y^2+f)p).
    \end{tikzcd}
\end{equation}
Hence one should not expect to gain insights about $\tilX$ which were not
already available through the pre-existing yoga of Kn\"orrer periodicity, up to
minor simplifications. Our result has more value in the other direction, for
understanding matrix factorizations of superpotentials containing terms of the
form $fq^2$ (see \cref{sec:res} for example).

\subsection{Applications}
We give various examples and applications in \cref{sec:ex}. For instance,
homological projective duality for a diagonal family of quadrics can be
interpreted using matrix factorizations for a superpotential of the form $\sum
f_iq_i^2$ representing the varying diagonal quadric equation. \cref{thm:intro}
suggests a geometric interpretation of this as an iterated double cover, which
we investigate in \cref{sec:diag}. Similar models relate to a recent mirror
symmetry proposal for singular double covers \cite[\S3.2]{LLRS}, and our
results may have implications for homological mirror symmetry in this context.

We also deduce an Orlov-type equivalence, relating the Kuznetsov component of a
Fano branched double cover of projective space to matrix factorizations for its
homogeneous discriminant on an affine orbifold (\cref{sec:orlov}). This is in
direct analogy with the classical result for Fano hypersurfaces in projective
space \cite{O2}, which can be understood as following from even Kn\"orrer
periodicity together with the theory of variation of GIT quotients \cite{Seg}.
We replace even Kn\"orrer periodicity by \cref{thm:intro}, applying the same
GIT construction.

As another example of the interaction of \cref{thm:intro} with variation of
GIT, we show in \cref{sec:exoflop} how it leads to a variant on the exoflop
construction of \cite{A} for a singular double cover. Exoflops relate
resolutions of singularities to partial compactifications of LG-models, which
in general give rise to categorical resolutions of singularities \cite{FK}. In
\cref{sec:res} we further apply \cref{thm:intro} to certain computations
involving categorical resolutions constructed in this way. We compute kernel
subcategories, obtaining results which match work using other constructions,
and exhibit an SOD relating two different partial compactifications. These
outcomes are consistent with optimistic conjectures on minimal categorical
resolutions.

\subsection*{Outline}
We discuss oddly-graded matrix factorizations in \cref{sec:odd}, prove the main
results in \cref{sec:thm}, and collect examples and applications in
\cref{sec:ex}. Everything is over $\C$.

\subsection*{Acknowledgements}
I must thank my advisor Ed Segal for keeping me sane in this world of
commutative, graded, non-commutative-graded rings, and much more.

This work was supported by the Engineering and Physical Sciences Research
Council [EP/S021590/1], the EPSRC Centre for Doctoral Training in Geometry and
Number Theory (the London School of Geometry and Number Theory), University
College London.

\section{Odd matrix factorizations} \label{sec:odd}

For a quasi-projective variety $X$ and a global function $w\in H^0(X,\O_X)$,
classical matrix factorizations are $\Z_2$-graded vector bundles
$\calE=\calE^0\oplus\calE^1$ equipped with $\Z_2$-graded morphisms
$d:\calE^\bullet\to(\calE[1])^\bullet\coloneqq\calE^{1-\bullet}$ such that
$d^2=w\cdot\id_{\calE}$. As for chain complexes, we can speak of chain
morphisms and chain homotopies for these objects, and when $X$ is affine the
homotopy category is the appropriate derived category.

When $X$ is non-affine, we must take sheaf cohomology into account \cite{LiP}.
This can be done by observing that the morphism sheaf $\shHom_X(\calE,\calF)$
for two matrix factorizations $\calE$ and $\calF$ again forms a matrix
factorization, but now for the difference of superpotentials $w-w=0$. It is a
$\Z_2$-graded chain complex of sheaves, and we can take sheaf cohomology:
\begin{equation} \label{eqn:locglob}
    \RHom(\calE,\calF) = \RG(X,\shHom_X(\calE,\calF)).
\end{equation}
This defines a suitable derived category.

Now suppose we have a non-trivial $\Z_2$-grading on $\O_X$, i.e. a
$\Z_2$-action on $X$. Matrix factorizations are then $\Z_2$-equivariant vector
bundles $\calE$, equipped with $\Z_2$-equivariant morphisms
\begin{equation*}
    d:\calE \to \calE[1] \coloneqq \calE\otimes\chi
\end{equation*}
satisfying $d^2=w\cdot\id_\calE$, where $\chi$ is the non-trivial character of
$\Z_2$. Locally these are $\Z_2$-graded projective modules over the
$\Z_2$-graded coordinate ring, and we recover the classical definition when the
$\Z_2$-action is trivial.

Once again, there are chain morphisms and chain homotopies for these objects,
from which we obtain a suitable homotopy category when $X$ is affine. However,
the above approach to defining $\RHom$ in the non-affine situation breaks down:
the differential on $\shHom_X(\calE,\calF)$ is no longer $\O_X$-linear in
general, so we cannot take derived global sections over $X$. This is because
when $\O_X$ includes odd elements it is typically \emph{not} super-commutative.
The differential on $\shHom_X(\calE,\calF)$ involves Koszul signs, which
interfere with $\O_X$-linearity.

\begin{example} \label[example]{ex:aff}
    Take $(X,w)=(\A^1,x^2)$ graded so that $x$ is odd, i.e. $\Z_2$ acts by
    $x\mapsto-x$. There is a matrix factorization given by the structure sheaf
    $\O_X$ equipped with the differential $d=x\cdot\id_{\O_X}$, and the induced
    differential on $\shHom_X(\O_X,\O_X)=\O_X$ is
    \begin{equation} \label{eqn:diff}
        x^n \mapsto d\circ x^n - (-1)^nx^n\circ d
            = \begin{dcases*}
                0 & $n$ even, \\
                2x^{n+1} & $n$ odd.
            \end{dcases*}
    \end{equation}
    This is not $\C[x]$-linear (although it is $\C[x^2]$-linear).
\end{example}

While this obstructs the computation of $\RHom$ using $\shHom_X$ on vector
bundles, we can still define a suitable derived category as the homotopy
category of matrix factorizations using injective sheaves instead of vector
bundles. This and other approaches are related through a universal construction
of the derived category of matrix factorizations as a Verdier quotient of the
homotopy category (with more general underlying sheaves) by a subcategory of
acyclic objects. Here acylic means homotopy equivalent to the totalization of
an exact sequence, which is a reformulation of the usual definition that can be
used even though $d^2\ne0$ \cite{P}, \cite{O3}.

This gives a suitable derived category of matrix factorizations even when $X$
is oddly-graded and non-affine, but we seemingly need injective resolutions to
compute morphisms. One of our main observations in this section is that
although morphisms can fail to form sheaves over $X$ when the grading is odd,
they do still form sheaves over $[X/\Z_2]$, allowing for local-to-global
computations like \cref{eqn:locglob} again.

\subsection{Definitions} \label{sec:def}
Here we fix our conventions for grading data, and record how standard theory
can be applied. An overview for the different types of input data and resulting
categories is given in \cref{fig:grading}.

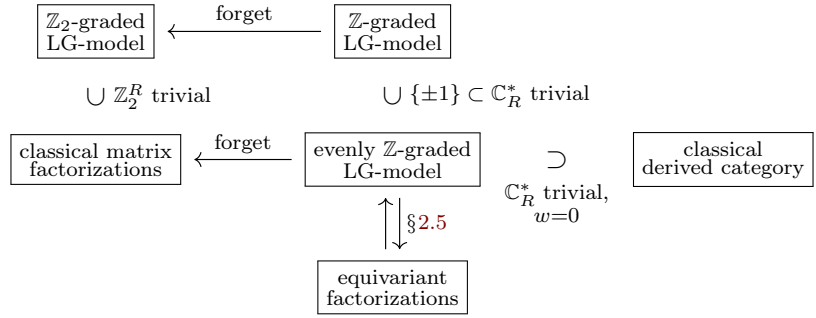
\begin{figure}[ht]
    \centering
    \begin{tikzcd}
        \boxed{\substack{\text{$\Z_2$-graded} \\ \text{LG-model}}} &[1em]
        \boxed{\substack{\text{$\Z$-graded} \\ \text{LG-model}}}
            \ar[l,"\text{forget}"'] \\
        \boxed{\substack{\text{classical matrix} \\ \text{factorizations}}}
            \ar[u,phantom,sloped,"\subset"]
            \ar[u,phantom,"\substack{\text{$\Z_2^R$ trivial}}"
                {xshift=2.3em}] &
        \boxed{\substack{\text{evenly $\Z$-graded} \\ \text{LG-model}}}
        \ar[l,"\text{forget}"']
            \ar[u,phantom,sloped,"\subset"]
            \ar[u,phantom,"\substack{\text{$\{\pm1\}\subset\C^*_R$ trivial}}"'
                {xshift=3.7em}]
            \ar[d,shift left,"\substack{\cref{sec:equi}}"] &[2em]
        \boxed{\substack{\text{classical} \\ \text{derived category}}}
            \ar[l,phantom,"\substack{\text{$\C^*_R$ trivial,} \\ w=0}"
                {yshift=-1.5em},"\supset"'] \\ &
        \boxed{\substack{\text{equivariant} \\ \text{factorizations}}}
            \ar[u,shift left]
    \end{tikzcd}
    \caption{Variations on grading data for Landau--Ginzburg models.}
    \label{fig:grading}
\end{figure}

\begin{definition}
    A \emph{$\Z_2$-graded Landau--Ginzburg model} consists of:
    \begin{enumerate}
        \item A smooth algebraic stack $X$.
        \item A $\Z_2$-action on $X$, called \emph{$R$-charge}.
        \item A $\Z_2$-invariant global function $w\in\Gamma(X,\O_X)$, called
            the \emph{superpotential}.
    \end{enumerate}
    It is \emph{evenly-graded} if $\Z_2$ acts trivially.
\end{definition}

\begin{definition}
    A \emph{$\Z$-graded Landau--Ginzburg model} consists of:
    \begin{enumerate}
        \item A smooth algebraic stack $X$.
        \item A $\C^*$-action on $X$, called \emph{$R$-charge}.
        \item A global function $w\in\Gamma(X,\O_X)$ of $\C^*$-weight 2, called
            the \emph{superpotential}.
    \end{enumerate}
    It is \emph{evenly-graded} if $-1\in\C^*$ acts trivially.
\end{definition}

Where disambiguation is desirable, we will write $\Z^R_2$ / $\C^*_R$ for the
groups acting as $R$-charge. There is a forgetful functor from $\Z$-graded to
$\Z_2$-graded LG-models which identifies $\Z^R_2$ with $\{\pm1\}\subset\C^*_R$.
We typically leave $R$-charge implicit, speaking of ``the LG-model $(X,w)$'' by
abuse of notation.

From now on we only consider $\Z$-graded LG-models, but everything can be
formulated identically at the $\Z_2$-graded level.

\begin{definition}
    A \emph{matrix factorization} on a $\Z$-graded LG-model $(X,w)$ consists of
    a $\C^*_R$-equivariant quasi-coherent sheaf $\calE$ equipped with a
    $\C^*_R$-equivariant morphism
    \begin{equation*}
        d : \calE \to \calE[1] \coloneqq \calE\otimes\chi
    \end{equation*}
    satisfying $d^2=w\cdot\id_\calE$, where $\chi$ is the generating character
    of $\C^*_R$.

    Given another matrix factorization $\calF$ on $(X,w)$, the morphism sheaf
    $\shHom_X(\calE,\calF)$ is again $\C^*_R$-equivariant, and so its global
    sections $\Hom(\calE,\calF)$ are $\Z$-graded. They inherit a square-zero
    differential
    \begin{equation} \label{eqn:homdiff}
        f \mapsto d_\calF\circ f - (-1)^{|f|}f\circ d_\calE,
    \end{equation}
    so matrix factorizations form a dg-category. It is pretriangulated, with
    shift functor\footnote{The effect on morphisms is $f\mapsto f\otimes\chi$
    with no sign.}
    \begin{equation} \label{eqn:shift}
        (\calE,d)[1] \coloneqq (\calE[1],-d).
    \end{equation}
    The mapping cones are the totalizations \cref{eqn:tot} of two-term
    complexes.
\end{definition}

\begin{remark} \label[remark]{rmk:wzero}
    Any algebraic stack $X$ forms a trivial $\Z$-graded LG-model where $\C^*_R$
    acts trivially and $w=0$. An equivariant sheaf for the trivial
    $\C^*_R$-action is simply a $\Z$-graded sheaf, so matrix factorizations for
    this LG-model are $\Z$-graded sheaves with degree 1 endomorphisms $d$
    satisfying $d^2=0$. This is the usual dg-category of chain complexes of
    sheaves on $X$.
\end{remark}

\begin{remark}
    If $(X,w)$ is evenly-graded, then the differential \cref{eqn:homdiff} on
    $\Hom(\calE,\calF)$ is induced by an $\O_X$-linear endomorphism of
    $\shHom_X(\calE,\calF)$, making the latter a matrix factorization for the
    LG-model $(X,0)$ with the same $R$-charge as $(X,w)$.
\end{remark}

\begin{notation}
    Typically one has a splitting of equivariant sheaves
    $\calE=\oplus_i\calE_i$, so then $d=\oplus_{i,j}d_{ij}$ where
    $d_{ij}:\calE_i\to\calE_j[1]$. Our notation for such an object is to
    display the sheaves $\calE_i$ with $d_{ij}$ written as an arrow from
    $\calE_i$ to $\calE_j$ when non-zero. Note the lack of shift; as displayed,
    the maps appear to be off-by-one in terms of $\C^*$-weight.
\end{notation}

\begin{example} \label[example]{ex:triv}
    There is always a trivial matrix factorization
    \begin{equation*}
        \left(\begin{tikzcd}
            \O[-1] \ar[r,shift left,"w"] & \O \ar[l,shift left,"1"]
        \end{tikzcd}\right),
    \end{equation*}
    which denotes the equivariant sheaf $\calE=\O[-1]\oplus\O$ with
    differential
    \begin{equation*}
        d = \begin{pmatrix}
            0 & 1 \\ w & 0
        \end{pmatrix} : \calE \to \calE[1].
    \end{equation*}
    This is contractible in the dg-category (homotopy-equivalent to zero) via
    the following contracting homotopy $h$ with $dh+hd=1$:
    \begin{equation*}
        \left(\begin{tikzcd}
            \O[-1] \ar[r,shift left,"1"] & \O \ar[l,shift left,"0"]
        \end{tikzcd}\right).
    \end{equation*}
    Note that $h$ has degree $-1$; we have displayed
    $h_{ij}:\calE_i\to\calE_j[-1]$ analogously to $d_{ij}$, with the missing
    shift now going in the opposite direction.
\end{example}

\begin{remark} \label[remark]{rmk:left}
    Locally the $\C^*_R$-action makes $\O_X$ a graded algebra, which we view as
    a formal curved dg-algebra with curvature $w$. This is possibly
    non-super-commutative if the grading is odd. Matrix factorizations
    correspond to (curved) \emph{right} dg-modules.

    The Koszul sign $f(rm)=(-1)^{|f||r|}rf(m)$ for left-linearity means that
    left dg-modules are not directly equivalent to $\C^*$-equivariant sheaves,
    and moreover that the shift functor on left dg-modules must modify scalar
    multiplication in odd degrees. Nevertheless, because $\O_X$ is commutative
    when ignoring gradings, it is isomorphic to its opposite dg-algebra by
    multiplication with $\sqrt{-1}$ in odd degrees. Hence left dg-modules give
    the same category (working over $\C$).
\end{remark}

\begin{definition}[\cite{P}, \cite{O3}, \cite{EP}, \cite{LiP}, \cite{BFK}, etc.]
    The \emph{(unbounded) derived category of quasi-coherent matrix
    factorizations} for the $\Z$-graded LG-model $(X,w)$ is the Drinfeld
    quotient of dg-categories
    \begin{equation*}
        \MF_\qc(X,w) \coloneqq
            \{\text{matrix factorizations}\}
                \; / \; \{\text{acyclic matrix factorizations}\},
    \end{equation*}
    where the \emph{acyclic} matrix factorizations are those that are
    homotopy-equivalent to totalizations of bounded exact sequences of matrix
    factorizations. By \emph{exact sequence}, we mean a sequence of closed
    morphisms\footnote{In practice we will more likely write an exact sequence
    with $f_i:\calE_i\to\calE_{i+1}$ of degree 0, but this can be translated
    into a sequence of degree 1 maps since $[1]$ is an autoequivalence.}
    \begin{equation*}
        f_i : \calE_i \to \calE_{i+1}[1]
    \end{equation*}
    such that the maps on underlying sheaves form an exact sequence in
    $\QCoh(X)$. If we write $d_i$ for the differential on $\calE_i$, then by
    \cref{eqn:shift} we have $d_{i+1}f_i=-f_id_i$, which combined with $d^2=w$
    and $f^2=0$ means that the \emph{totalization}
    \begin{equation} \label{eqn:tot}
        \bigl(\oplus_i\calE_i, \;
            {\textstyle\sum_i(d_i+f_i)}\bigr)
    \end{equation}
    is again a matrix factorization of $w$.

    Weak equivalences in the derived category are called
    \emph{quasi-isomorphisms}.
\end{definition}

\begin{remark}
    For the trivial LG-model $(X,0)$ this recovers the usual (unbounded)
    derived category of quasi-coherent sheaves $\MF_\qc(X,0)=D_\qc(X)$ via
    \cref{rmk:wzero}.
\end{remark}

\begin{remark}
    As usual, we conflate the dg-category $\MF_\qc(X,w)$ with its triangulated
    homotopy category. The latter is equivalent to the homotopy category of
    injective matrix factorizations \cite[Proposition 2.4]{LiP}, and if $X$ is
    affine also to the homotopy category of projective matrix fatorizations
    \cite[Proposition 3.14]{BFK}. If $X$ has enough vector bundles, and is
    evenly-graded, it is equivalent to a category of vector bundles where we
    incorporate sheaf cohomology into the morphisms via \cref{eqn:locglob}, as
    in \cite{Seg}.
\end{remark}

\begin{definition}
    We write $\MF(X,w)\subset\MF_\qc(X,w)$ for the thick closure of the
    subcategory of \emph{coherent matrix factorizations}, i.e. those with
    coherent underlying sheaf. These are the compact objects in $\MF_\qc(X,w)$
    \cite[Proposition 2.17]{LiP}.
\end{definition}

\begin{remark}
    For the trivial LG-model $(X,0)$ we have $\MF(X,0)=D^b(X)$.
\end{remark}

From now on we will mostly work with the bounded category $\MF(X,w)$, typically
using the term ``matrix factorization'' to refer to coherent matrix
factorizations.

\begin{example} \label[example]{ex:koszul}
    Suppose $w=fg$, with $f$ a non-zerodivisor. We then have a short exact
    sequence of matrix factorizations
    \begin{equation*}
        \biggl(\begin{tikzcd}[column sep=small]
            \O[-n] \ar[r,shift left,"w"] & \O[1-n] \ar[l,shift left,"1"]
        \end{tikzcd}\biggr) \xrightarrow{\begin{pmatrix}
            f & 0 \\ 0 & 1
        \end{pmatrix}}
        \biggl(\begin{tikzcd}[column sep=small]
            \O \ar[r,shift left,"g"] & \O[1-n] \ar[l,shift left,"f"]
        \end{tikzcd}\biggr) \xrightarrow{\begin{pmatrix}
            1 & 0
        \end{pmatrix}}
        \O/f,
    \end{equation*}
    where $|f|=n$ is the $R$-charge weight. Note that $\O/f$ is a matrix
    factorization with differential equal to zero, which is valid since
    $w|_{\{f=0\}}=0$. The leftmost term is contractible by \cref{ex:triv}, so
    the totalization is homotopy equivalent to the mapping cone for the
    morphism
    \begin{equation*}
        \biggl(\begin{tikzcd}[column sep=small]
            \O \ar[r,shift left,"g"] & \O[1-n] \ar[l,shift left,"f"]
        \end{tikzcd}\biggr) \xrightarrow{\begin{pmatrix}
            1 & 0
        \end{pmatrix}} \O/f.
    \end{equation*}
    Hence this is a quasi-isomorphism, giving a Koszul resolution of $\O/f$ in
    $\MF(X,w)$. It generalizes: \cite[Proposition 3.20]{BFK},
    \cite[\S3.2.1]{ASS}.
\end{example}

\subsection{Local morphisms} \label{sec:loc}
We now explain the existence of an almost-internal Hom functor
$\RshHom_{[X/\Z_2]}$ replacing $\RshHom_X$ when the grading is odd. This
follows from two basic facts:
\begin{enumerate}
    \item The induced grading on $[X/\Z_2]$ is concentrated in even degrees.
    \item The structure sheaf $\O_X$ is a coherent sheaf of curved algebras on
        $[X/\Z_2]$.
\end{enumerate}
Since $\O_{[X/\Z_2]}$ is then super-commutative, we can apply standard theory
as in \cite{EP} to the sheaf $\O_X$ of curved algebras over it, reconstructing
$\MF(X,w)$ as a category linear over $[X/\Z_2]$, i.e. with a well-defined
$\RshHom_{[X/\Z_2]}$.

By induced grading, we mean the $\C^*_R$-action on $[X/\Z_2]$ which is
descended from the original action on $X$. This makes the quotient map
$\pi:X\to[X/\Z_2]$ equivariant, so $\pi_*\O_X$ is a sheaf of graded algebras
over the evenly-graded structure sheaf $\O_{[X/\Z_2]}$.

\begin{remark}
    Here $\O_{[X/\Z_2]}$ is evenly-graded because $-1\in\C^*_R$ acts
    trivially on $[X/\Z_2]$ by construction. This can be understood
    concretely via the smooth atlas
    \begin{equation} \label{eqn:atlas}
        [X/\Z_2] = [(X\times\C^*)/\C^*_G],
    \end{equation}
    where $\C^*_G$ acts via $R$-charge on $X$ and with weight $-2$ on
    $\C^*$.\footnote{Here $\C^*_G=\C^*$ is named purely to distinguish it from
    the group acting as $R$-charge.} The $\C^*_R$-action on $[X/\Z_2]$ can be
    given by an action on $X\times\C^*$ which is trivial on $X$ and has weight
    2 on $\C^*$, since this differs from $R$-charge on $X$ by the
    $\C^*_G$-action which we quotient by. Hence we have an explicit even-degree
    cohomological grading $|z|=2$ on $\O_{X\times\C^*}=\O_X[z^{\pm1}]$, with
    $\pi_*\O_X$ corresponding to $\O_{X\times\C^*}[\sqrt{z}\,]$ with
    $|\sqrt z|=1$.

    Of course, everything on $[(X\times\C^*)/\C^*_G]$ is also
    $\C^*_G$-equivariant, which corresponds to another grading on
    $\O_{X\times\C^*}$, but this is not the cohomological grading.
    Differentials have degree 0 with respect to it, and it is not used for
    Koszul signs.
\end{remark}

\begin{example}
    Consider again $(X,w)=(\A^1,x^2)$, with $|x|=1$. Here $\pi_*\O_X$ is the
    sheaf of algebras $\O\oplus\O(1)$ on $[\A^1/\Z_2]$, where $\O(1)$ is the
    2-torsion line bundle pulled back from $B\Z_2$. The matrix factorization
    from \cref{ex:aff} corresponds to
    \begin{equation*}
        \left(\begin{tikzcd}
            \O(1) \ar[r,shift left,"x"] & \O \ar[l,shift left,"x"]
        \end{tikzcd}\right)
    \end{equation*}
    viewed as a right-linear curved module over $\pi_*\O_X=\O\oplus\O(1)$. Its
    endomorphisms, as computed in \cref{eqn:diff}, form the complex
    \begin{equation} \label{eqn:homs}
        \left(\begin{tikzcd}
            \O(1) \ar[r,shift left,"2x"] & \O \ar[l,shift left,"0"]
        \end{tikzcd}\right)
            \in D^b_{\C^*_R}([\A^1/\Z_2]) \coloneqq \MF([\A^1/\Z_2],0).
    \end{equation}
    After taking global sections, i.e. pushing forward to the coarse moduli
    space, this recovers the natural interpretation of \cref{eqn:diff} as a
    complex of $\C[x^2]$-modules.
\end{example}

\begin{remark}
    Note that in \cref{eqn:homs} the induced $R$-charge on $[\A^1/\Z_2]$ is
    non-trivial, and so $\MF([\A^1/\Z_2],0)$ is not the same as the usual
    category $D^b([\A^1/\Z_2])$. However, because $R$-charge is already elided
    in the $\MF(X,w)$ notation, we will actually treat $D^b(X)$ and $D_\qc(X)$
    as synonyms for $\MF(X,0)$ and $\MF_\qc(X,0)$, where there is implicitly a
    potentially non-trivial $R$-charge. This should cause no confusion, since
    if $X$ has a non-trivial $R$-charge we will never consider the trivially
    $R$-charged $D^b(X)$.

    In fancy terms, this is the derived category for $X$ viewed as a derived
    algebraic stack, where the structure sheaf is a formal dg-algebra via the
    $R$-charge grading.
\end{remark}

For completeness, and because related ideas will be important later, we outline
what formally goes into the argument for existence of $\RshHom_{[X/\Z_2]}$.

\begin{defprop}
    Let $X$ be a smooth algebraic stack with a $\C^*_R$-action. A sheaf of
    curved algebras over $X$ is an algebra $\calA$ in the monoidal category of
    $\C^*_R$-equivariant sheaves on $X$ with a curvature $w\in\Gamma(X,\calA)$
    of $\C^*_R$-weight 2.

    There is a derived category of quasi-coherent sheaves of curved
    \emph{right}\footnote{\cref{rmk:left}.} $\calA$-modules,
    $\MF_\qc(X,\calA,w)$, and of coherent $\calA$-modules, $\MF(X,\calA,w)$,
    defined as in \cite[\S1]{EP}. When $\O_X$ is evenly-graded (or more
    generally, when it is super-commutative) homomorphisms of $\calA$-modules form
    sheaves on $X$, i.e. we have a functor
    \begin{equation*}
        \RshHom_{\calA/X} : \MF_\qc(X,\calA,w) \otimes \MF_\qc(X,\calA,w)
            \to D_\qc(X)
    \end{equation*}
    such that
    \begin{equation*}
        \RHom_\calA(\calE,\calF) = \RG(X,\RshHom_{\calA/X}(\calE,\calF)).
    \end{equation*}
\end{defprop}

Note that \cite[\S1]{EP} covers the more general setting of sheaves of curved
dg-algebras, but without a grading on $\O_X$. It is straightforward to include
such a grading.

\begin{lemma} \label[lemma]{lem:rel}
    Suppose $(X,w)$ is a $\Z$-graded Landau--Ginzburg model, and $\pi:X\to Y$
    is an affine morphism which is equivariant with respect to a given
    $\C^*_R$-action on $Y$. Then $\pi_*$ induces an equivalence
    \begin{equation*}
        \MF(X,w)\simeq\MF(Y,\pi_*\O_X,w).
    \end{equation*}
    If $Y$ is evenly-graded, there is therefore a relative $\shHom$ functor
    \begin{equation*}
        \RshHom_{X/Y} : \MF_\qc(X,w)\otimes\MF_\qc(X,w) \to D_\qc(Y)
    \end{equation*}
    such that
    \begin{equation*}
        \RHom_X(\calE,\calF) = \RG(Y,\RshHom_{X/Y}(\calE,\calF)).
    \end{equation*}
\end{lemma}

This follows from (very classical) equivalences of underived module categories.
Applying it to the finite morphism $\pi:X\to[X/\Z_2]$ gives
\begin{equation*}
    \RshHom_{[X/\Z_2]} : \MF_\qc(X,w)\otimes\MF_\qc(X,w) \to D_\qc([X/\Z_2]),
\end{equation*}
and if $X$ is evenly-graded this is related to the usual $\RshHom_X$ by
\begin{equation*}
    \RshHom_{[X/\Z_2]}(\calE,\calF) = \pi_*\RshHom_X(\calE,\calF).
\end{equation*}

\begin{remark}
    Since $\pi$ is finite, not just affine, we have
    \begin{equation*}
        \RshHom_{[X/\Z_2]} : \MF(X,w)\otimes\MF(X,w) \to D^b([X/\Z_2]),
    \end{equation*}
    provided that $X$ has enough vector bundles. This follows from resolving
    coherent matrix factorizations by finite type vector bundles, as in
    \cite[Lemma 3.6]{Seg}.
\end{remark}

\begin{remark}
    Since $w$ descends to $[X/\Z_2]$, there is also the category
    $\MF([\A^1/\Z_2],w)$. This has an involution $\O(1)\otimes-$ coming from
    the 2-torsion line bundle $\O(1)$ pulled back from $B\Z_2$, and the
    equivariant category for this involution gives an alternative construction
    of $\MF(X,w)$ \cite{E}. It is equivalent to analyzing
    $\MF([X/\Z_2],\pi_*\O_X,w)$ via an explicit description of the algebra
    structure on the rank 2 sheaf $\pi_*\O_X$.
\end{remark}

\subsection{Even models} \label{sec:even}

Because $\pi:X\to[X/\Z_2]$ is a (stacky) double cover, we will be able to apply
\cref{thm:main} to it. This turns out to produce an evenly-graded LG-model
which is equivalent to the oddly-graded LG-model that we started with.

More precisely, we have a Cartesian diagram
\begin{equation*}
    \begin{tikzcd}[row sep=small,column sep=small]
        X \ar[d] \ar[r] & \pt \ar[d] \\
        {[X/\Z_2]} \ar[r,"\rho"] & B\Z_2.
    \end{tikzcd}
\end{equation*}
The double cover $\pt\to B\Z_2$ corresponds to the 2-torsion line bundle
$\O(1)$ on $B\Z_2$, and hence $X\to[X/\Z_2]$ corresponds to $\rho^*\O(1)$,
which we also write simply as $\O(1)$. Applying \cref{thm:main} to this
2-torsion line bundle (see \cref{rmk:triv}) gives
\begin{equation} \label{eqn:evenpre}
    \MF(X,w) \simeq \MF(\Tot\{\O(1)\to[X/\Z_2]\},w+q^2).
\end{equation}
More explicitly, since $\O(1)$ is pulled back from $B\Z_2$ we can write
\begin{equation*}
    \Tot\{\O(1) \to [X/\Z_2]\}
        = [X/\Z_2] \times_{B\Z_2} \Tot\{\O(1)\to B\Z_2\}
        = [(X\times\A^1)/\Z_2],
\end{equation*}
where $\Z_2$ acts by $R$-charge on $X$ and by $q\mapsto-q$ on $\A^1$. Then
\cref{eqn:evenpre} becomes
\begin{equation} \label{eqn:even}
    \MF(X,w) \simeq \MF([(X\times\A^1)/\Z_2],w+q^2),
\end{equation}
where the right-hand side has the original $R$-charge on $X$, with $|q|=1$.
This precisely matches the $\Z_2$-action, making the new LG-model
evenly-graded.

\begin{example}
    Consider $(\A^2,x^2)$ with $R$-charge $|x|=1$, $|y|=0$. Then
    \cref{eqn:even} gives
    \begin{equation*}
        \MF(\A^2,x^2) \simeq \MF([\A^2_{x,q}/\Z_2]\times\A^1_y,x^2+q^2).
    \end{equation*}
    This is equivalent to $B\Z_2\times\A^1_y$ by even Kn\"orrer periodicity, so
    we get
    \begin{equation*}
        \MF(\A^2,x^2) \simeq D^b(B\Z_2\times\A^1) \simeq D^b(\A^1)^{\oplus2}
    \end{equation*}
    which also follows from \cref{thm:main} for the trivial double cover of
    $\A^1$.
\end{example}

\begin{example}
    Now consider a different $R$-charge on $(\A^2,x^2)$, where $|x|=|y|=1$.
    Here \cref{eqn:even} becomes
    \begin{equation*}
        \MF(\A^2,x^2) \simeq \MF([\A^3/\Z_2],x^2+q^2),
    \end{equation*}
    and Kn\"orrer periodicity gives $[\A^1/\Z_2]$ instead of $B\Z_2\times\A^1$.
    This is not the trivial double cover of $\A^1$; \cref{thm:main} fails to
    apply because $y$ was odd (see \cref{sec:tech}(\ref{itm:tech2})).
\end{example}

What happens here is that the objects $\begin{tikzcd} \O \ar[loop
right,distance=1em,"\pm x"] \end{tikzcd}\in\MF(\A^2,x^2)$ are orthogonal if $y$
is even, but not if $y$ is odd. This is because odd degree morphisms of matrix
factorizations are closed if they anti-commute with the differential, by
\cref{eqn:homdiff}, so multiplication by $y$ is a closed map between the two
objects iff $y$ is odd. These objects correspond to the two branches of the
trivial double cover $\A^1\amalg\A^1$ when $y$ is even and the two orbifold
line bundles on $[\A^1/\Z_2]$ when $y$ is odd.

\subsection{Loop factorizations}

One main novelty with oddly-graded LG-models is the possibility of rank 1
factorizations of the form $\begin{tikzcd}\O_Z\ar[loop right,distance=1em,"{\pm
f}"]\end{tikzcd}$, where say $f\in\Gamma(\O_Z)$ is a non-zerodivisor and
$w|_Z=f^2$. It can help to understand these objects as summands in the
splitting of a Koszul resolution (\cref{ex:koszul}), diagonalizing its
differential:
\begin{equation} \label{eqn:split}
    \O_Z/f = \left(\begin{tikzcd}
            \O_Z \ar[r,shift left,"f"] & \O_Z \ar[l,shift left,"f"]
    \end{tikzcd}\right) =
    \left(\begin{tikzcd}
        \O_Z \ar[loop right,distance=1em,"f"]
    \end{tikzcd}\right) \oplus
    \left(\begin{tikzcd}
        \O_Z \ar[loop right,distance=1em,"-f"]
    \end{tikzcd}\right).
\end{equation}

For example, on $(\A^1,x^2)$ we can compute without Koszul signs that
\begin{equation*}
    \RshHom(\O/x,\,\O/x) = \RshHom\bigl(\hspace*{-0.4em}\begin{tikzcd}
        \O \ar[r,shift left,"x"] & \O \ar[l,shift left,"x"]
    \end{tikzcd}\hspace*{-0.4em}, \, \O/x\bigr) = \O/x \oplus \O/x,
\end{equation*}
from which we deduce using \cref{eqn:split} that
\begin{align*}
    \RshHom\bigl(\hspace*{-0.4em}\begin{tikzcd}
        \O \ar[loop right,distance=1em,"\pm x"]
    \end{tikzcd},\hspace*{-0.2em}\begin{tikzcd}
        \O \ar[loop right,distance=1em,"\pm x"]
    \end{tikzcd}\hspace*{-0.1em}\bigr) &= \O/x, \\
    \RshHom\bigl(\hspace*{-0.4em}\begin{tikzcd}
        \O \ar[loop right,distance=1em,"\pm x"]
    \end{tikzcd},\hspace*{-0.2em}\begin{tikzcd}
        \O \ar[loop right,distance=1em,"\mp x"]
    \end{tikzcd}\hspace*{-0.1em}\bigr) &= 0.
\end{align*}
These also follow from direct computations with Koszul signs like
\cref{eqn:diff}. The right-hand sides of these equations should be viewed as
sheaves on $[\A^1/\Z_2]$; \cref{sec:loc}.

\subsection{Equivariant factorizations} \label{sec:equi}

Here we explain the relationship between matrix factorizations in our sense and
equivariant factorizations as formulated in \cite{BFK}. In short, the latter
framework is equivalent to working with evenly-graded matrix factorizations,
and cannot be directly applied in the oddly-graded case.

Following \cite[\S3]{BFK}, suppose that $G$ is an affine algebraic group acting
on a smooth variety $X$, and $w\in\Gamma(X,L)$ is a $G$-invariant section of a
$G$-equivariant line bundle $L$. There is then a dg-category of equivariant
factorizations, which we will refer to as $\EF(X,G,w)$, whose objects have the
form
\begin{equation} \label{eqn:equi}
    \begin{tikzcd}[column sep=small]
        \calE_{-1} \ar[r,"f"] & \calE_0 \ar[r,"g"] & \calE_{-1}\otimes L
    \end{tikzcd}
\end{equation}
where everything is $G$-equivariant and $gf=w$. We will be sloppy about the
distinction between (bounded and unbounded) derived and underived categories in
what follows, but it is not hard to see that all variants match up on both
sides.

To relate $(X,G,w)$ to an LG-model in our sense, we pass to the $\C^*$-bundle
\begin{equation*}
    U \coloneqq \Tot\{(L^{-1})^\times \xrightarrow{\pi} [X/G]\},
\end{equation*}
which is the complement of the zero-section in $\Tot\{L^{-1}\to[X/G]\}$. Let
$\C^*_R$ act fiberwise on $U$ with weight 2. There is a tautological section
$p\in\Gamma(U,\pi^*(L^{-1})^\times)$, and the equivariant factorization
\cref{eqn:equi} gives rise to an object
\begin{equation*}
    \left(\begin{tikzcd}
        (\pi^*\calE_{-1})[1] \ar[r,shift left,"\pi^*f"] &[1em]
        \pi^*\calE_0 \ar[l,shift left,"(\pi^*g)p"]
    \end{tikzcd}\right) \in \MF(U,(\pi^*w)p).
\end{equation*}
Note that $\pi^*$ canonically produces $\C^*_R$-equivariant sheaves, because
$\pi$ is equivariant with respect to the trivial $\C^*_R$-action on $[X/G]$.
Also, note that $U$ is evenly-graded. This gives a functor
$\EF(X,G,w)\to\MF(U,(\pi^*w)p)$.

Conversely, suppose $(X,w)$ is an evenly-graded LG-model. Because it is
evenly-graded, the $\C^*_R$-action factors through the quotient
$\C^*_G\coloneqq\C^*_R/\{\pm1\}$. Since $\Z^R_2$ acts trivially, any matrix
factorization $(\calE,d)$ splits into $(\pm1)$-eigensheaves
$\calE=\calE_+\oplus\calE_-$ interchanged by $d$. Viewing these as
$\C^*_G$-equivariant sheaves, we get
\begin{equation*}
    \calE_- \xrightarrow{d} \calE_+[1]
        \xrightarrow{d} \calE_-[2] = \calE_-\otimes\chi,
\end{equation*}
where $\chi$ is the generating character of $\C^*_G$. Note that
$w\in\Gamma(X,\O(\chi))^{\C^*_G}$. This gives a functor
$\MF(X,w)\to\EF(X,\C^*_G,w)$, which is easily seen to be an equivalence.

Applying this to the evenly-graded LG-model $(U,(\pi^*w)p)$ from above, we find
\begin{equation*}
    \MF(U,(\pi^*w)p) \simeq \EF(U,\C^*_G,(\pi^*w)p)
\end{equation*}
where $\C^*_G$ acts on $U$ fiberwise with weight 1. By \cite[Lemma 3.48]{BFK}
we then have
\begin{equation*}
    \MF(U,(\pi^*w)p) \simeq \EF(U,\C^*_G,(\pi^*w)p) \simeq \EF(X,G,w),
\end{equation*}
which is realized by the functor $\EF(X,G,w)\to\MF(U,(\pi^*w)p)$ constructed
above.

This shows that equivariant factorizations can be written as matrix
factorizations on evenly-graded LG-models, and vice versa. Hence the framework
of equivariant factorizations is essentially equivalent to the framework of
evenly-graded LG-models. Oddly-graded matrix factorizations cannot be written
as equivariant factorizations, since they do not split into pairs
$(\calE_0,\calE_{-1})$ in general.

\begin{example}
    Consider $(X,G,w)=(\A^1,\C^*_G,x^2)$ where $\C^*_G$ acts with weight 1.
    Then
    \begin{equation*}
        (U,(\pi^*w)p) = \bigl([(\A^1\times\C^*)/\C^*_G], \, x^2p\bigr)
    \end{equation*}
    where $\C^*_G$ acts with weight $-2$ on $\C^*$, and $R$-charge is $|x|=0$,
    $|p|=2$. As in \cref{eqn:atlas}, this is equivalent to $([\A^1/\Z_2],x^2)$
    with $|x|=1$. Hence
    \begin{equation*}
        \EF(\A^1,\C^*_G,x^2) \simeq \MF([\A^1/\Z_2],x^2) \simeq D^b(\pt).
    \end{equation*}
    More generally, for any $\Z$-graded LG-model $(X,w)$ the equivariant
    factorization category $\EF(X,\C^*_R,w)$ is equivalent to
    $\MF([X/\Z^R_2],w)$. When $X$ is evenly-graded we can replace $\C^*_R$ by
    the quotient $\C^*_R/\{\pm1\}$, which instead gives $\MF(X,w)$.
\end{example}

\section{Main theorem} \label{sec:thm}

Firstly, here is a rough outline. Notation as in \cref{sec:stmt}, we have a
diagram
\begin{equation}
    \begin{tikzcd}
        (L^{-1},fq^2) \ar[d,"\pi"'] &
        (\tilL^{-1}\coloneqq\tau^*L^{-1},fq^2) \ar[l,"\tau"'] \ar[d,"\pi"'] \\
        X & \tilX \coloneqq \{y^2=f\} \subset L, \ar[l,"\tau"']
    \end{tikzcd}
\end{equation}
and the equivalence $\MF(L^{-1},fq^2)\simeq D^b(\tilX)$ is given by the
Fourier--Mukai kernel
\begin{equation} \label{eqn:fm}
    \left(\begin{tikzcd}
        \O_{\tilL^{-1}} \ar[loop right,distance=2em,"yq"]
    \end{tikzcd}\right) \in \MF(\tilL^{-1},fq^2).
\end{equation}
When $X=\pt$ and $\tilX=\{y^2=1\}$ this recovers the two generators
\begin{equation*}
    \left(\begin{tikzcd}
        \C[q] \ar[loop right,distance=1em,"\pm q"]
    \end{tikzcd}\right)\in\MF(\A^1,q^2)
\end{equation*}
at the two points $y=\pm1$.

\subsection{Technicalities} \label{sec:tech}

The full statement will replace $X$ by an LG-model $(X,w)$, adding both a
superpotential and a non-trivial $R$-charge, which should be compatible with
$L$ and $f$. The above description then breaks down for two reasons:

\begin{enumerate}[itemsep=5pt]
    \item \label{itm:tech1}
        When $L^{-1}$ and $\tilX$ are both oddly-graded, the Fourier--Mukai
        kernel should be a module over the tensor product of the two
        oddly-graded algebras, which is non-commutative; there are non-trivial
        Koszul signs. This means we have to phrase everything algebraically, as
        the geometric fiber product $\tilL^{-1}$ does not have these Koszul
        signs.

    \item \label{itm:tech2}
        If the base $X$ itself is oddly-graded, it turns out the equivalence is
        false as stated. This is again due to missing Koszul signs, and the
        correct general statement replaces $L^{-1}$ and $\tilX$ by
        non-commutative algebras where odd fiber coordinates anti-commute with
        odd functions on $X$.\footnote{We insert Koszul signs into
        $\O_{L^{-1}[1]}=\Sym^\bullet_X(L[-1])$, forming an algebra in
        $\O_X$-bimodules locally given by $R[q]$ with product
        $rq^asq^b=(-1)^{a|q||s|}rsq^{a+b}$, and similarly for
        $\O_\tilX=\Sym^\bullet_X(L^{-1})/(y^2-f)$.} The necessity of these
        signs is illustrated in \cref{ex:kp}.
\end{enumerate}

In what follows we assume that $X$ is evenly-graded, ignoring the subtleties of
(\ref{itm:tech2}). The alternative would mean dealing with more general
non-commutative LG-models, and we are only really interested in commutative
applications.

\subsection{Involutions} Before getting into the full statement and proof, we
make an observation about the geometry of this Fourier--Mukai transform.

\begin{proposition} \label[proposition]{prop:inv}
    The deck transformation $\sigma:y\mapsto-y$ on $\tilX$ corresponds to the
    involution $\rho:q\mapsto-q$ on $L^{-1}$ under the Fourier--Mukai transform
    from \cref{eqn:fm}.
\end{proposition}

\begin{proof}
    Write $\tilsigma$ and $\tilrho$ for the lifts of $\sigma$ and $\rho$ to the
    roof $\tilL^{-1}$. Letting $\calF\in D^b(\tilX)$, we compute
    \begin{align*}
        \rho^*\tau_*(\pi^*\sigma^*\calF\otimes(\begin{tikzcd}
            \O \ar[loop right,distance=1em,"yq"]
        \end{tikzcd}))
            &= \tau_*\tilrho^*\tilsigma^*(\pi^*\calF\otimes(\begin{tikzcd}
                \O \ar[loop right,distance=1em,"-yq"]
            \end{tikzcd})) \\
            &= \tau_*\tilsigma^*\tilrho^*(\pi^*\calF\otimes(\begin{tikzcd}
                \O \ar[loop right,distance=1em,"-yq"]
            \end{tikzcd})) \\
            &= \tau_*(\pi^*\calF\otimes(\begin{tikzcd}
                    \O \ar[loop right,distance=1em,"yq"]
            \end{tikzcd})).
    \end{align*}
    The first isomorphism follows from $\rho\tau=\tau\tilrho$ and
    $\sigma\pi=\pi\tilsigma$, while the second and third use
    $\tilrho\tilsigma=\tilsigma\tilrho$, $\tau\tilsigma=\tau$ and
    $\pi\tilrho=\pi$.
\end{proof}

We can then take $\Z_2$-equivariant categories \cite{E} on both sides of the
equivalence
\begin{equation*}
    \MF(L^{-1},fq^2) \simeq D^b(\tilX),
\end{equation*}
which is equivalent to taking the stacky quotients by $\rho$ and $\sigma$. This
gives the following corollary to \cref{thm:main}, as stated in the
introduction:

\begin{corollary} \label[corollary]{cor:root}
    $\MF([L^{-1}/\Z_2],fq^2)\simeq D^b([\tilX/\Z_2])=D^b(\sqrt{f/X})$.
\end{corollary}

We have restricted this statement to the setup where $X$ has no superpotential
or $R$-charge, but it almost certainly remains true in the more general context
giving
\begin{equation*}
    \MF([L^{-1}/\Z_2],w+fq^2) \simeq \MF(\sqrt{f/X},w).
\end{equation*}
The analogue of \cref{prop:inv} becomes more subtle, due to the technicality
\cref{sec:tech}(\ref{itm:tech1}), so we have omitted the details of this
generalization.

\subsection{Statement and proof}

Let $(X,w)$ be an LG-model, and take a $\C^*_R$-equivariant line bundle $L$ on
$X$ with a $\C^*_R$-invariant section $f\in\Gamma(X,L^2)$.

The total space of $L$ inherits a $\C^*_R$-action which restricts to
$\tilX=\{y^2=f\}\subset L$ because $y\in\Gamma(L,\tau^*L)$ is
$\C^*_R$-invariant, so we get an LG-model $(\tilX,\tau^*w)$.

Similarly, the total space of the $\C^*_R$-equivariant line bundle
$L^{-1}[1]\coloneqq L^{-1}\otimes\chi$ inherits a $\C^*_R$-action, and
$q\in\Gamma(L^{-1}[1],\pi^*L^{-1}[1])$ is a $\C^*_R$-invariant section, giving
another LG-model $(\Tot(L^{-1}[1]),\pi^*w+fq^2)$. The shift ensures $fq^2$ has
$\C^*_R$-weight 2.

\begin{theorem} \label[theorem]{thm:main}
    If $(X,w)$ is evenly-graded and $\tilX$ is smooth, there is an equivalence
    \begin{equation*}
        \MF(\Tot(L^{-1}[1]), \, \pi^*w+fq^2) \simeq \MF(\tilX,\tau^*w).
    \end{equation*}
\end{theorem}

\begin{remark}
    The assumption that $\tilX$ is smooth is not essential; we include it
    because we only considered LG-models for smooth stacks in \cref{sec:def}.
    When $\tilX$ is non-smooth the equivalence still holds (with coherent
    matrix factorizations, not perfect matrix factorizations), but there is one
    caveat that if the $\C^*_R$-action on $L^{-1}[1]$ is fiberwise trivial
    then we must restrict to a formal neighbourhood of the zero section in
    $L^{-1}[1]$. When $\tilX$ is smooth the critical locus $\Crit(fq^2)$ is
    supported on the zero section, so this is unnecessary. The same caveat
    applies to the usual even Kn\"orrer periodicity equivalence for a singular
    hypersurface.
\end{remark}

\begin{remark} \label[remark]{rmk:triv}
    Unbranched double covers occur when $f$ trivializes $L^2$, i.e. when $L$ is
    a 2-torsion line bundle. Often the choice of trivialization has no impact,
    and we might suppress it from the notation, writing $f=1$. The root stack
    is then just $\sqrt{\emptyset\,/\,X}=X$, and \cref{cor:root} is a family
    version of the equivalence $\MF([\A^1/\Z_2],x^2)\simeq D^b(\pt)$.
\end{remark}

\begin{proof}[Proof of \cref{thm:main}]
    Using \cref{lem:rel}, we phrase everything in terms of the sheaves of
    curved algebras
    \begin{equation*}
        A = \pi_*\O_{\Tot(L^{-1}[1])} = \Sym^\bullet_X(L[-1])
            \quad \text{and} \quad
        B = \tau_*\O_\tilX = \O_X\oplus L^{-1}.
    \end{equation*}
    The coordinates $q$ and $y$ on $L^{-1}$ and $\tilX$ induce morphisms
    $L_q:A\to L^{-1}\otimes_XA$ and $R_y:B\to B\otimes_XL$ which are right
    $A$-linear and left $B$-linear of degrees 1 and 0. Our proposed equivalence
    between $\MF(X,A,\pi^*w+fq^2)$ and $\MF(X,B,\tau^*w)$ is the functor
    $-\otimes_BM$ given by the curved $(B,A)$-bimodule
    \begin{equation*}
        M = (B\otimes_XA, \; R_y\otimes L_q).
    \end{equation*}
    From $\otimes$-$\Hom$ adjunction this has a right adjoint given by the
    curved $(A,B)$-bimodule
    \begin{equation*}
        N = \Hom_A(M,A) = (A\otimes_XB^{\vee_X}, \; -R_q\otimes R_y^\vee).
    \end{equation*}
    The unit and counit of the adjunction give us bilinear maps $B\to
    M\otimes_AN$ and $N\otimes_BM\to A$ where $A$ and $B$ are the respective
    diagonal bimodules. We will verify that these are quasi-isomorphisms of
    curved bimodules.

    A priori we get functors at the level of $\MF_\qc$, with $-\otimes_BM$
    descending to the level of $\MF$ because $M$ is coherent as a right
    $A$-module. While $N$ is not coherent as a right $B$-module, the fact that
    it is an inverse equivalence for $M$ will imply that it does still preserve
    the subcategory $\MF\subset\MF_\qc$ of compact objects.

    For the unit, we have
    \begin{align*}
        M\otimes_AN &= \bigl(B\otimes_XA\otimes_XB^{\vee_X}, \;
            R_y\otimes L_q\otimes1 - 1\otimes R_q\otimes R_y^\vee\bigr) \\
                &= \left(\begin{tikzcd}
                        [column sep=large,ampersand replacement=\&]
                B\otimes_XB^{\vee_X}
                \ar[r,"R_y\otimes1-1\otimes R_y^\vee"] \&[2em]
                B\otimes_XB^{\vee_X}\otimes_XL[-1] \\
                \qquad\qquad \ar[r,"R_y\otimes1-1\otimes R_y^\vee"] \&
                B\otimes_XB^{\vee_X}\otimes_XL^2[-2] \ar[r] \&[-2em] \cdots
            \end{tikzcd}\right).
    \end{align*}
    Expanding the algebra structure on $B=\O_X\oplus L^{-1}$ one can compute
    directly that this complex is exact after the first term, and that the
    kernel in degree 0 is isomorphic to $B$ via the map $B\to
    B\otimes_XB^{\vee_X}$; $b\mapsto b(1\otimes1^\vee+y\otimes y^\vee)$. This
    is the unit of the adjunction, corresponding to multiplication
    $B\otimes_XB\to B$, which is $B$-bilinear.

    For the counit, we have
    \begin{align*}
        N\otimes_BM
            &= \bigl(A\otimes_XB^{\vee_X}\otimes_XA, \;
            -R_q\otimes R_y^\vee\otimes1+1\otimes L_y^\vee\otimes L_q\bigr) \\
            &= \left(\begin{tikzcd}[column sep=huge,ampersand replacement=\&]
                A\otimes_XA \ar[r,shift left,"f(1\otimes L_q-R_q\otimes1)"]
                    \&[2em]
                A\otimes_XA\otimes_XL
                    \ar[l,shift left,"1\otimes L_q-R_q\otimes1"]
        \end{tikzcd}\right).
    \end{align*}
    This is a Koszul-type factorization (\cref{ex:koszul}) of $f(1\otimes
    q^2-q^2\otimes1)$, and hence is quasi-isomorphic to $\coker(1\otimes
    L_q-R_q\otimes1)$ via projection from $A\otimes_XA$. This cokernel is
    isomorphic to the diagonal bimodule $A$ via the multiplication map
    $A\otimes_XA\to A$. Again, one can check that projection to $A\otimes_XA$
    followed by multiplication is the counit of the adjunction. It is clearly
    $A$-bilinear.
\end{proof}

\begin{example} \label[example]{ex:kp}
    There is an analogous statement for $\Z_2$-graded LG-models, where
    everything is $\Z_2^R$-equivariant instead of $\C^*_R$-equivariant. We can
    use this version to explore classical $\Z_2$-graded Kn\"orrer periodicity
    as an example of \cref{sec:tech}(\ref{itm:tech2}).

    Indeed, start with the trivial $\Z_2$-graded LG-model $(\Spec\C,0)$. The
    trivial line bundle $L$ with $f=1$ gives an equivalence
    \begin{equation*}
        \MF(\C[q],q^2) \simeq D^b(\C[y]/(y^2-1)),
    \end{equation*}
    where the two choices of $\Z_2^R$-equivariant structure on $L$ determine
    whether $q$ is odd and $y$ is even or vice versa. (Here $D^b(-)$ denotes
    the $\Z_2$-graded derived category.)

    For classical even Kn\"orrer periodicity we want to consider
    $\MF(\C[p,q],p^2+q^2)$ where $p$ and $q$ are even. Again, there is an
    equivalence
    \begin{equation*}
        \MF(\C[p,q],p^2+q^2) \simeq \MF(\C[p,y]/(y^2-1),p^2)
    \end{equation*}
    where $y$ is odd, because $q$ was even. We want to view this second
    LG-model as an application of the theorem to $\Spec\C[y]/(y^2-1)$ with the
    trivial line bundle again, giving a double cover of $\Spec\C[y]/(y^2-1)$.
    Because $y$ is odd, by \cref{sec:tech}(\ref{itm:tech2}) we should introduce
    a Koszul sign in this double cover. The result is the following:
    \begin{equation*}
        \MF(\C[p,y]/(y^2-1),p^2)
            \simeq D^b(\C\langle y,z\rangle/(y^2-1,z^2-1,yz+zy)).
    \end{equation*}
    This ``iterated double cover'' is now a non-commutative algebra, which is
    in fact the Clifford algebra for a rank two quadratic form. It is Morita
    trivial, which is a restatement of classical Kn\"orrer periodicity. If we
    had instead used the na\"ive iterated double cover we would get a
    commutative algebra which is very much not Morita trivial, violating the
    expected Kn\"orrer periodicity equivalence.

    The correct way to realize this commutative iterated double cover is to
    make $y$ and $z$ even, by taking $p$ and $q$ to be odd. Then $p$ and $q$
    have to anti-commute:
    \begin{equation*}
        \MF(\C\langle p,q\rangle/(pq+qp),p^2+q^2)
            \simeq D^b(\C[y,z]/(y^2-1,z^2-1))
            \simeq D^b(\pt)^{\oplus4}.
    \end{equation*}
    This non-commutative LG-model with odd coordinates $p$ and $q$ is
    equivalent to the category of Fourier--Mukai kernels for functors
    $\MF(\C[p],p^2)\to\MF(\C[q],q^2)$, where the signs are an instance of
    \cref{sec:tech}(\ref{itm:tech1}). Its equivalence with $D^b(\pt)^{\oplus4}$
    is as expected from $\MF(\C[p],p^2)\simeq D^b(\pt)^{\oplus2}$.
\end{example}

\section{Examples} \label{sec:ex}

\subsection{Mirror symmetry} \label{sec:mirr}
It turns out that oddly-graded LG-models arise naturally in some basic
1-dimensional examples of mirror symmetry. Recall the well-known mirror pair of
toric LG-models $([\A^1/\Z_n],x^n)$ and $(\A^1,x^n)$. In one direction, there
is a $\Z$-graded homological mirror symmetry equivalence
\begin{equation} \label{eqn:hmsa}
    \MF([\A^1/\Z_n],x^n) \simeq D^b(A_{n-1}) \simeq \W(\A^1,x^n),
\end{equation}
where $\W(\A^1,x^n)$ is the partially wrapped Fukaya category of a disc with
$n$ stops, equivalent to the $A_{n-1}$ quiver by \cite[\S6.2]{HKK}, and the
$B$-side LG-model $([\A^1/\Z_n],x^n)$ is evenly $\Z$-graded with
$|x|=\frac{2}{n}$, and was shown to be equivalent to the $A_{n-1}$ quiver
(using slightly different language) in \cite{T}.

\begin{remark}
    The length $n-1$ exceptional collection on $\MF([\A^1/\Z_n],x^n)$ can be
    deduced from variation of GIT: the orbifold $[\A^1/\Z_n]$ is one of the
    phases in the GIT problem $\C^2/\C^*$ where $\C^*$ acts with weights
    $(1,-n)$, and the superpotential extends as $x^ny$. The other phase is
    $(\A^1,y)$, and $\MF(\A^1,y)=0$ since there are no critical points, so the
    semi-orthogonal decomposition coming from variation of GIT \cite{BFK2}
    gives a full exceptional collection on $\MF([\A^1/\Z_n],x^n)$ with length
    $n-1$ coming from the sum of the weights.
\end{remark}

If we swap the $A$-side and $B$-side, there should be an equivalence
\begin{equation*}
    \MF(\A^1,x^n) \simeq \W([\A^1/\Z_n],x^n).
\end{equation*}
This should hold for $\Z_2$-graded categories with no grading data, but we can
also ask about $\Z$-graded versions. On the $B$-side we must have
$|x|=\frac{2}{n}$, so there is an even $\Z$-grading if $n=1$, an odd
$\Z$-grading if $n=2$, and no $\Z$-grading for $n\ge3$. On the $A$-side we have
a disc with one stop\footnote{The superpotential $x^n$ gives $n$ stops on the
covering space, and hence one stop on the $\Z_n$-quotient.} and one $\Z_n$
orbifold point. If $n=1$ this is smooth, $\Z$-gradeable by the obvious line
field on the disc. For $n\ge2$ we have an orbifold (we will use \cite{BSW} as a
reference for partially wrapped Fukaya categories of orbifold surfaces; see
also \cite[\S9]{Sei}). Line fields only exist over orbifold points of order 2,
so again there is a $\Z$-grading for $n=2$ and no $\Z$-grading for $n\ge3$. The
line field has winding number 1 around a $\Z_2$ orbifold point, so it reduces
to a non-trivial $\Z_2$-grading on the Fukaya category, matching the odd
grading on the $B$-side. See \cite[\S2.2]{LeP} for a more detailed discussion
of grading data for Fukaya categories.

When $n=1$ both categories are trivial, but when $n=2$ we have a non-trivial
equivalence between the oddly-graded $B$-side LG-model $(\A^1,x^2)$ and the
partially wrapped Fukaya category of a disc with one orbifold point and one
stop. Both of these categories are generated by two orthogonal exceptional
objects:
\begin{equation*}
    \MF(\A^1,x^2) \simeq D^b(\pt)^{\oplus2} \simeq \W([\A^1/\Z_2],x^2).
\end{equation*}
The first equivalence is the basic example from the present paper, while the
second is in \cite[\S4]{BSW}. More generally, the $n$-stopped
$\W([\A^1/\Z_2],x^{2n})$ from \cite[\S4]{BSW} is mirror to the
oddly-graded\footnote{The $\Z_n$-quotient ensures that the fractional
$\C^*$-action $\lambda\cdot x=\lambda^{1/n}x$ is well-defined, but it does not
ensure that $-1\in\C^*$ acts trivially. For that we would need a
$\Z_{2n}$-quotient, which would bring us back to the type $A$ model.} $B$-side
model $\MF([\A^1/\Z_n],x^{2n})$ where $|x|=\frac{1}{n}$:
\begin{equation*}
    \MF([\A^1/\Z_n],x^{2n}) \simeq D^b(D_{n+1}) \simeq \W([\A^1/\Z_2],x^{2n}).
\end{equation*}
The equivalence between the matrix factorization category and partially wrapped
Fukaya category here can be deduced from the type $A$ equivalence
\cref{eqn:hmsa} by taking $\Z_2$-equivariant categories \cite{E} corresponding
to the stacky quotients $[\A^1/\Z_n]\to[\A^1/\Z_{2n}]$ and
$\A^1\to[\A^1/\Z_2]$. The equivalence with the $D_{n+1}$ quiver for the
$A$-side follows from \cite[Corollary 4.13]{BSW},\footnote{See also
\cite[Remark 4.15]{BSW} regarding quadratic relations.} from which we can
deduce it for the $B$-side also.

\begin{remark}
    The full exceptional collection on $\MF([\A^1/\Z_n],x^{2n})$ coming from
    the equivalence with $D^b(D_{n+1})$ can again be deduced from variation of
    GIT. Indeed, $([\A^1/\Z_n],x^{2n})$ is a phase of the GIT problem
    $\C^2/\C^*$ with weights $(1,-n)$ and superpotential $x^{2n}y^2$, and hence
    has a sequence of $n-1$ exceptional objects with orthogonal complement
    equivalent to $\MF(\A^1,y^2)$. This consists of two further exceptional
    objects which are mutually orthogonal, giving the $D_{n+1}$ quiver.

    Note that here we are applying the standard theory of variation of GIT to
    oddly-graded LG-models. The results almost certainly remain true in this
    setting, with the same proofs, but we will not check the details. See the
    following sections for further examples.
\end{remark}

\subsection{Orlov-type equivalences} \label{sec:orlov}

A hypersurface $\{f=0\}\subset\P^{n-1}$ is related via semi-orthogonal
decomposition to the orbifold LG-model $([\A^n/\Z_d],f)$ by a seminal result of
Orlov \cite{O2}. This result fits into the framework of GLSMs, in that it can
be deduced by analyzing the different phases of the GIT problem $\C^{n+1}/\C^*$
with weights $(1,\ldots,1,-d)$ and superpotential $fp$ \cite{Seg}. The affine
orbifold is one such phase, while the other phase is the line bundle
$(\Tot\O(-d)_{\P^{n-1}},fp)$ which is equivalent to $D^b(\{f=0\})$ by Kn\"orrer
periodicity \cite{H}.

As \cref{thm:main} gives a dimensional reduction of Kn\"orrer periodicity for
double covers \cref{eqn:sqrt}, we also get a dimensional reduction of Orlov's
theorem in the special case of a double cover. Again we will apply standard
results for variation of GIT quotients to oddly-graded LG-models, although this
specific application can also be deduced by an alternative argument
(\cref{rmk:orltwo}).

Take a double cover $X\xrightarrow{2:1}\P^{n-1}$, with branch locus
$\{g=0\}\subset\P^{n-1}$ of degree $2d$. Applying \cref{thm:main} gives
\begin{equation*}
    D^b(X) \simeq \MF(\Tot\O(-d)_{\P^{n-1}},gq^2),
\end{equation*}
where $\O(-d)_{\P^{n-1}}$ is a phase of the GIT problem $\C^{n+1}/\C^*$ with
weights $(1,\ldots,1,-d)$ and $R$-charges $(0,\ldots,0,1)$. The other phase is
the affine orbifold $([\A^n/\Z_d],g)$ with $\Z_d$-weights $(1,\ldots,1)$ and
$R$-charges $(\frac{1}{d},\ldots,\frac{1}{d})$. When $d\le n$ we then have
\begin{equation*}
    D^b(X) = \langle \MF([\A^n/\Z_d],g),E_1,\ldots,E_{n-d}\rangle
\end{equation*}
for certain exceptional objects $E_i$, and a reversed decomposition when $d\ge
n$ \cite{BFK2}.

Note that $[\A^n/\Z_d]$ is oddly-graded here, because the $\C^*$-action
$\lambda\cdot x=\lambda^{1/d}x$ is such that $-1\in\C^*$ acts non-trivially (it
is not cancelled by an element of $\Z_d$). This is in contrast with the
original hypersurface version of Orlov's theorem, which involves the even
grading $|x|=\frac{2}{d}$ on $([\A^n/\Z_d],f)$.

\begin{remark} \label[remark]{rmk:orltwo}
    Applying \cref{sec:even}, this oddly-graded affine orbifold LG-model is
    equivalent to an evenly-graded one, which is an instance of a weighted
    version of the original hypersurface form of Orlov's theorem. Specifically,
    we have
    \begin{equation*}
        \MF([\A^n/\Z_d],g) \simeq \MF([\A^{n+1}/\Z_{2d}],y^2+g)
    \end{equation*}
    where the $\Z_{2d}$-weights are $(1,\ldots,1,d)$. Variation of GIT relates
    this (via semi-orthogonal decomposition) to the weighted projective
    hypersurface
    \begin{equation*}
        \{y^2 + g = 0\} \subset \P^{[1:\cdots:1:d]},
    \end{equation*}
    which is simply another construction of the double cover $X$ branched over
    $g$. Hence the semi-orthogonal decomposition we talk about in this section
    can be understood using standard results for evenly-graded LG-models one
    dimension up.
\end{remark}

\begin{example}
    By writing a quadric as a double cover, we see that the Kuznetsov component
    $\calA_Q$ for a $d$-dimensional quadric $Q$ should have an equivalence
    \begin{equation*}
        \calA_Q \simeq \MF(\A^d,x_0^2+\cdots+x_d^2),
    \end{equation*}
    where $|x_i|=1$. By even Kn\"orrer periodicity this reduces to two cases:
    $\calA_Q\simeq D^b(\pt)$ if $d$ is odd, and
    $\calA_Q\simeq\MF(\A^1,x^2)\simeq D^b(\pt)^{\oplus2}$ if $d$ is even.
    Compare \cref{eqn:fiber}, which is the evenly-graded analogue of this
    computation using $\Z_2$ orbifold models.
\end{example}

\subsection{Diagonal quadric families} \label{sec:diag}

Matrix factorizations of quadric families over a $\Z_2$-gerbe are a
well-established object of study with links to mirror symmetry, non-commutative
resolutions and homological projective duality \cite{ASS}, \cite{K1}.

The fiberwise behaviour is understood from two examples:
\begin{align} \label{eqn:fiber}
\begin{split}
    \MF([\A^{2n}/\Z_2],x_1^2+\cdots+x_{2n}^2)
        \simeq \MF([\A^2/\Z_2],x^2+y^2) &\simeq D^b(\pt)^{\oplus2}, \\
    \MF([\A^{2n+1}/\Z_2],x_1^2+\cdots+x_{2n+1}^2)
        \simeq \MF([\A^1/\Z_2],x^2) &\simeq D^b(\pt).
\end{split}
\end{align}
Generically even-rank families behave like 2-to-1 covers with branching over
the corank 1 locus, while odd rank families behave like 1-to-1 covers with
doubling over the corank 1 locus. In fact, up to a Brauer twist, they are
non-commutative resolutions of the double cover branched over the discriminant
divisor (even rank) or the associated root stack (odd rank).

\begin{example}
    Typical examples relating to complete intersections of quadrics are
    families over even-weighted projective space, but the essential behaviour
    can be understood in affine charts. One possible local model is
    \begin{equation*}
        ([\A^2_{x,y}/\Z_2]\times\A^3_{a,b,c},ax^2+bxy+cy^2),
    \end{equation*}
    giving a non-commutative resolution of the 3-fold ordinary double point
    viewed as a double cover branched over $\{b^2=4ac\}\subset\A^3$. The rank
    drops by 2 at the singularity, and one finds families of matrix
    factorizations parametrized by the exceptional divisors of the two crepant
    resolutions \cite[\S5]{ASS}.
\end{example}

Our result gives a description of these non-commutative resolutions for
diagonal quadric families in terms of iterated double covers, modulo some
complications coming from technicality \cref{sec:tech}(\ref{itm:tech2}), which
require us to enlarge the group acting.

We consider the LG-model
\begin{equation*}
    \bigl([\A^n/G]\times\A^n, \; \textstyle\sum_ia_ix_i^2\bigr),
\end{equation*}
where $a_i$ and $x_i$ are coordinates on $\A^n$ and $[\A^n/G]$ respectively,
$G=\Z_2^n$ if $n$ is odd, and $G=\ker(\Z_2^n\to\Z_2)$ if $n$ is even. If $G$
was just $\Z_2$, this would be the expected local model. It is related to the
iterated double cover
\begin{equation*}
    Y\coloneqq\cap_i\{y_i^2=a_i\}\subset\A^{2n},
\end{equation*}
which is acted on by $G$ via the $y_i$ coordinates. The quotient stack $[Y/G]$
is a crepant orbifold resolution of the singular double cover
$\{y^2=\prod_ia_i\}$ if $n$ is even, and of the associated root stack if $n$ is
odd.

\begin{proposition} \label[proposition]{prop:iter}
    There is an equivalence
    \begin{equation*}
        \MF\bigl([\A^n/G]\times\A^n, \; \textstyle\sum_ia_ix_i^2\bigr)
            \simeq D^b([Y/G]).
    \end{equation*}
\end{proposition}

\begin{proof}
    This is nothing more than an iteration of \cref{thm:main}. We must check
    that the base is evenly-graded at each intermediate step, because of
    \cref{sec:tech}(\ref{itm:tech2}). This involves for each $0\le k<n$ the
    following spaces:
    \begin{equation*}
        \begin{tikzcd}[column sep=huge]
            \cap_{i\le k}\{y_i^2=a_i\} \subset
            \bigl(\bigl[(\A^{n-k}\times\A^k)/G\bigr]\times\A^n, \;
            \textstyle\sum_{i\le n-k}a_ix_i^2\bigr)
                \ar[d] \\
            \cap_{i\le k}\{y_i^2=a_i\} \subset
            \bigl(\bigl[(\A^{n-k-1}\times\A^k)/G\bigr]\times\A^n, \;
                \textstyle\sum_{i\le n-k-1}a_ix_i^2\bigr) \\
            \cap_{i\le k+1}\{y_i^2=a_i\} \subset
            \bigl(\bigl[(\A^{n-k-1}\times\A^{k+1})/G\bigr]\times\A^n, \;
                \textstyle\sum_{i\le n-k-1}a_ix_i^2\bigr). \ar[u]
        \end{tikzcd}
    \end{equation*}
    The $R$-charge weights are $|x_i|=1$ and $|y_i|=|a_i|=0$, and it suffices
    to show that there is a $\Z_2\subset G$ which acts with these weights on
    the middle space. When $G=\Z_2^n$ this is clear, and the extra constraint
    $\Z_2\subset\ker(\Z_2^n\to\Z_2)$ for even $n$ is satisfiable since we are
    acting on $\A^{n-k-1}\times\A^k=\A^{n-1}$ so one of the weights in $\Z_2^n$
    is arbitrary.
\end{proof}

\begin{remark}
    One can globalize this to the setting where the $a_i$ are sections of the
    squares of certain line bundles, and $G$ acts fiberwise on the total space
    of the sum of these line bundles. This corresponds to a trivial
    $\Z_2$-gerbe structure on the base, so some twisting is required to recover
    examples over even-weighted projective space.
\end{remark}

\begin{remark}
    The diagonal $\Z_2\subset\Z_2^n$ is contained in $G$, with equality for
    $n\in\{1,2\}$.
\end{remark}

\begin{conjecture}
    In the above setup, we have an equivalence
    \begin{equation*}
        \MF\bigl([\A^n/G]\times\A^n, \; \textstyle\sum_ia_ix_i^2\bigr)
            \simeq
            \MF\bigl([\A^n/\Z_2]\times\A^n, \; \textstyle\sum_ia_ix_i^2\bigr).
    \end{equation*}
\end{conjecture}

This is expected since \cref{prop:iter} shows that both are crepant resolutions
of the same singularity, even though geometrically the equivalence is
unexpected. The morphism induced by $\Z_2\subset G$ is \emph{not} an
equivalence. Fiberwise we have
\begin{equation*}
    \MF\bigl([\A^n/G],\textstyle\sum_ix_i^2\bigr)
        \simeq D^b\bigl([\cap_i\{y_i^2=1\}/G]\bigr) = \begin{dcases*}
            D^b(\pt\amalg\pt) & $n$ even \\
            D^b(\pt) & $n$ odd,
        \end{dcases*}
\end{equation*}
but unlike \cref{eqn:fiber} the generators are not given by maximal isotropic
subspaces.

\begin{example}
    As a particular example of this fiberwise equivalence, we have
    \begin{equation*}
        \MF([\A^2/\Z_2^2],x^2+y^2) \simeq D^b([\{\pm1\}^2/\Z_2^2]) = D^b(\pt).
    \end{equation*}
    We could view this as a family of even rank quadratic forms over $B\Z_2^2$,
    which even Kn\"orrer periodicity suggests should be equivalent to
    $D^b(B\Z_2^2)\simeq D^b(\pt)^{\oplus4}$. This fails because there is a
    non-trivial Brauer twist, due to the lack of a $\Z_2^2$-invariant maximal
    isotropic subspace. In other words, this $\Z_2^2$-invariant quadratic form
    corresponds to a non-zero Brauer class $\alpha\in H^2(\Z_2^2,\C^*)=\Z_2$
    with
    \begin{equation*}
        \MF([\A^2/\Z_2^2],x^2+y^2)
            \simeq D^b(B\Z_2^2,\alpha) \simeq D^b(\pt).
    \end{equation*}
    It turns out that $\Coh(B\Z_2^2,\alpha)\simeq\Coh(\pt)$.
\end{example}

\begin{example}
    In the case $n=2$ we have $G=\Z_2$, and the model
    \begin{equation*}
        \bigl([\A^2_{x,y}/\Z_2]\times\A^2_{a,b},ax^2+by^2\bigr)
    \end{equation*}
    should be a non-commutative resolution for the double cover
    $\{c^2=ab\}\subset\A^3$. Applying \cref{prop:iter}, we get an equivalence
    \begin{equation*}
        \MF\bigl([\A^2_{x,y}/\Z_2]\times\A^2_{a,b},ax^2+by^2\bigr)
            \simeq D^b\bigl(\bigl[\A^2_{\sqrt a,\sqrt b}/\Z_2\bigr]\bigr).
    \end{equation*}
    This latter space is the (crepant) orbifold resolution of the quotient
    singularity.
\end{example}

\begin{remark}
    Hodge number duality for a recently-proposed mirror family of singular
    double covers was shown in \cite{HL}, with a proof based on establishing a
    Hodge-theoretical link between $(L^{-1},fq^2)$ and $\tilX$, and making use
    of the same iterated double covers mentioned above. The mirror symmetry
    proposals here make use of a LG-model with a superpotential of the form
    $\sum_if_iq_i^2$ \cite[\S3.2]{LLRS}, and it seems likely that our results
    could have implications for homological mirror symmetry in this context
    (e.g. \cite[Conjecture 4.4]{LLR}).
\end{remark}

\subsection{Exoflops} \label{sec:exoflop}

The term exoflop was coined in \cite{A} in reference to a phenomenon where
applying GLSM techniques (Kn\"orrer periodicity and GIT) to a resolution of
singularities leads to a partial compactification of LG-models.

For an example, consider the nodal surface $X=\{xy=z^2\}\subset\A^3$ and its
minimal resolution $\tilX=\{xy=z^2\}\subset\Bl_\pt\A^3=\O(-1)_{\P^2}$.
Kn\"orrer periodicity gives us equivalent LG-models:
\begin{equation*}
    \begin{tikzcd}
        X \ar[r,"\simeq"] & (\A^4,(xy-z^2)p) \\[-0.8em]
        \tilX \ar[u] \ar[r,"\simeq"] &
        \bigl([\O(-1)\oplus\O(-2)]_{\P^2},(xy-z^2)p\bigr). \ar[u,dashed]
    \end{tikzcd}
\end{equation*}
Now $[\O(-1)\oplus\O(-2)]_{\P^2}$ is one of the phases of a GIT problem
$\C^5/\C^*$ with weights $(1,1,1,-1,-2)$ and $R$-charges $(0,0,0,0,2)$. The
other phase is an orbifold $\O(-1)^3_{\P^{[1:2]}}$, and this is a Calabi--Yau
problem so we get an equivalence
\begin{equation*}
    D^b(\tilX) \simeq \MF\bigl([\O(-1)\oplus\O(-2)]_{\P^2},(xy-z^2)p\bigr)
        \simeq \MF\bigl(\O(-1)^3_{\P^{[1:2]}},(xy-z^2)p\bigr).
\end{equation*}
The key observation is that $(\O(-1)^3_{\P^{[1:2]}},(xy-z^2)p)$ partially
compactifies the original model $(\A^3,(xy-z^2)p)$, which is an affine chart
missing the orbifold point. Resolution becomes (partial) compactification: the
exceptional divisor has ``flopped out''.

On the other hand, we can view $X$ as a double cover of $\A^2$ branched over
$\{xy=0\}$, and $\tilX$ as a double cover of $\Bl_\pt\A^2$ branched over its
strict transform. We then get equivalences with oddly-graded LG-models from
\cref{thm:main}:
\begin{equation*}
    \begin{tikzcd}
        X \ar[r,"\simeq"] & (\A^3,xyq^2) \\[-0.8em]
        \tilX \ar[u] \ar[r,"\simeq"] &
        \bigl(\O(-1)^2_{\P^1},xyq^2\bigr). \ar[u,dashed]
    \end{tikzcd}
\end{equation*}
Again the resulting model $\O(-1)^2_{\P^1}$ has a flop, which can be realized
via GIT. Here there are no orbifolds; it is simply the Atiyah flop. We get an
equivalence
\begin{equation*}
    D^b(\tilX) \simeq \MF(\O(-1)^2_{\P^1_{x:y}},xyq^2)
        \simeq \MF(\O(-1)^2_{\P^1_{s:q}},xyq^2),
\end{equation*}
and again this final space is a partial compactification of the original model
$(\A^3,xyq^2)$. The two compactifications are related by another application of
\cref{thm:main}; we illustrate these equivalences in \cref{fig:exoflop}.

\begin{figure}[ht]
    \centering
    \begin{equation*}
        \begin{tikzcd}[column sep=tiny]
            (\A^4,(xy-z^2)p) \ar[r,dashed,bend left,"\text{compactify}"] &
            (\O(-1)^3_{\P^{[1:2]}},(xy-z^2)p) &
            ([\O(-1)\oplus\O(-2)]_{\P^2},(xy-z^2)p)
                \ar[l,<->,bend right,"\text{flop}"',"\simeq"] \\
            X \ar[u,"\text{even Kn\"orrer}"',"\simeq"]
                \ar[d,"\text{\cref{thm:main}, }z^2=xy","\simeq"'] & &
            \tilX \ar[u,"\text{even Kn\"orrer}"',"\simeq"]
                \ar[d,"\text{\cref{thm:main}, }z^2=xy","\simeq"'] \\
            (\A^3,xyq^2) \ar[r,dashed,bend right] &
            (\O(-1)^2_{\P^1_{s:q}},xyq^2)
                \ar[uu,"\text{\cref{thm:main}, }q^2=p"',"\simeq"] &
            (\O(-1)^2_{\P^1_{x:y}},xyq^2). \ar[l,<->,bend left,"\simeq"']
        \end{tikzcd}
    \end{equation*}
    \caption{Nodal exoflops.} \label{fig:exoflop}
\end{figure}
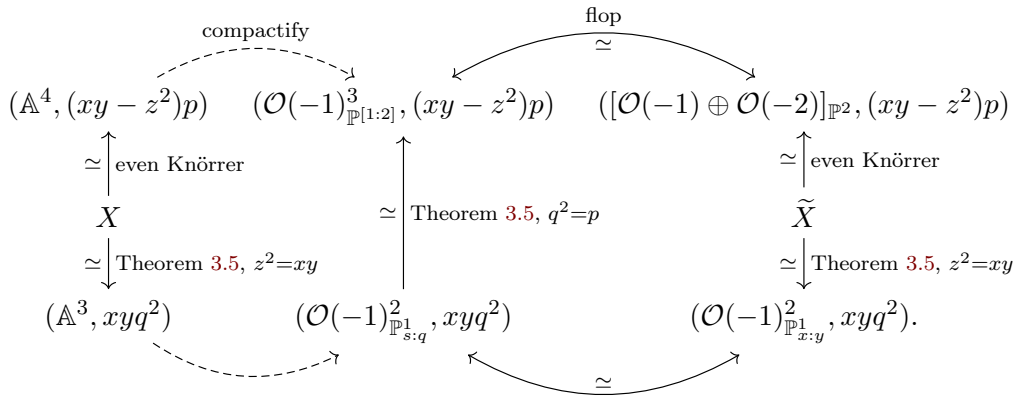

\subsection{Categorical resolutions} \label{sec:res}

Exoflops show that resolutions can give partial compactifications of LG-models.
We can actually reverse this arrow, constructing categorical resolutions of
complete intersection singularities by compactifying their associated LG-models
\cite{FK}. This builds categorical resolutions \emph{ex nihilo}, and the
results are a priori unrelated to any geometric resolution. In practice,
however, we find categories that are equivalent to constructions from \cite{K2}
and \cite{KL} which relate to $D^b(\tilX)$ via semi-orthogonal decomposition.
This is as would be predicted by optimistic conjectures on ``minimal''
categorical resolutions.

More specifically, suppose $X=\{f=0\}\subset S$ is a singular hypersurface in a
smooth ambient space $S$. From Kn\"orrer periodicity, we get an LG-model
$(\O_S(-X),fp)$ with critical locus $\Crit(fp)=X\cup\O_S(-X)|_{\Sing(X)}$. Now,
it is known that $\MF(U,w)$ is homologically smooth and proper iff $\Crit(w)$
is proper \cite{FK}, and when $X$ is proper we should enforce that a
categorical resolution is homologically smooth and proper. Hence, we should
partially compactify the LG-model $(\O_S(-X),fp)$ so that the $\A^1$ fibers of
$\O_S(-X)|_{\Sing(X)}\subset\Crit(fp)$ are compactified to projective lines,
possibly with orbifold structure. Given such a partial compactification,
pushforward and pullback along the open embedding of the original model form
the categorical structure of a resolution \cite{FK}. When $X$ is not proper,
the same process accomplishes a relative properness over $X$ which again
procludes pathological categorical resolutions, allowing us to work with
non-proper singularities \cite[\S2]{C}. This fits the philosophy that
resolution of singularities should be a local process.

\begin{example} \label[example]{ex:dual}
    Consider $\C[x]/x^n$, which is derived equivalent to $(\A^2,x^np)$ by
    Kn\"orrer periodicity. A suitable compactification is given by the orbifold
    model
    \begin{equation*}
        (\O(-1)_{\P^{[1:n]}_{s:p}},x^np),
    \end{equation*}
    which is a GIT quotient $\C^3/\C^*$ with weights $(-1,1,n)$ and $R$-charges
    $(0,0,2)$. Here the affine fiber over the singularity has been compactified
    to an orbifold projective line $\P^{[1:n]}$. The other phase of this GIT
    problem is an affine orbifold $([\A^2/\Z_n],p)$, which gives a trivial
    matrix factorization category since $\Crit(p)$ is empty. This means that
    $\MF(\O(-1)_{\P^{[1:n]}},x^np)$ has a full exceptional collection
    \cite{BFK2}, and in fact it is derived equivalent to the Auslander algebra
    for $\C[x]/x^n$ which has a full exceptional collection from
    \cite[\S5.4]{KL}.
\end{example}

The quadratic case $n=2$ in \cref{ex:dual} can be understood using
\cref{thm:main}. Indeed, the categorical resolution
\begin{equation*}
    \MF(\O(-1)_{\P^{[1:2]}},x^2p)
\end{equation*}
is equivalent to the double cover of $\P^{[1:2]}$ branched over $\{p=0\}$. This
double cover is simply $\P^1$ with induced $R$-charge weights $(0,1)$, with the
Beilinson full exceptional collection. We can view this oddly-graded $\P^1$ as
compactifying the oddly-graded $\A^1$ which comes out of Koszul duality for the
dual numbers: $D^b(\C[x]/x^2)\simeq D^b(\C[\theta])$, where $|\theta|=1$. This
instance of Koszul duality can also be deduced from \cref{thm:main}:
\begin{equation*}
    D^b(\C[x]/x^2) \simeq \MF(\A^2,x^2p) \simeq D^b(\C[\sqrt p\,]),
        \qquad |\sqrt p|=1.
\end{equation*}

Orbifold points in the partial compactification usually indicate non-minimality
of the resolution via the analogy of root stacks with blowups (e.g.
\cite[\S4.2]{C}, \cref{sec:A_3}), but this example shows that appearances can
be misleading: the orbifold point is cancelled by a quadratic term. This also
happened in \cref{sec:exoflop}.

\subsubsection{Type $A_2$} \label{sec:A_2}

Consider the $A_2$ curve singularity $\{y^2=x^3\}\subset\A^2$. We will use
\cref{thm:main} to compute the kernel subcategory of a categorical resolution
of this singularity constructed using compactifications of LG-models. The
result of the computation matches the analysis of \cite{F}, where categorical
resolutions of higher-dimensional singularities of type $A_2$ were constructed
using the methods of \cite{K2}.

By Kn\"orrer periodicity, we can start with the LG-model $(\A^3,(y^2-x^3)p)$.
The categorical resolution we will consider is given by the partial
compactification
\begin{equation*}
    (\O(-1)^2_{\P^{[1:2]}},(y^2-x^3s)p),
\end{equation*}
where $x,y$ are fiber coordinates on the line bundles, and $s,p$ are
homogeneous coordinates on $\P^{[1:2]}$. The kernel subcategory is the
subcategory of matrix factorizations supported on the boundary divisor of the
compactification, which is $\{s=0\}$. To compute this we can restrict to the
open neighbourhood $\{p\ne0\}$, which is isomorphic to the affine orbifold
LG-model
\begin{equation*}
    ([\A^3/\Z_2],y^2-x^3s)
\end{equation*}
with $\Z_2$-weights $(1,1,1)$ and $R$-charges $(1,1,-1)$. Applying
\cref{thm:main}, we have
\begin{equation*}
    \MF([\A^3/\Z_2],y^2-x^3s) \simeq \MF(\A^2,-x^3s),
\end{equation*}
with $\A^2$ viewed as a stacky double cover of $[\A^2/\Z_2]$. The kernel
subcategory again corresponds to matrix factorizations supported at $\{s=0\}$.
Finally, even Kn\"orrer periodicity gives
\begin{equation*}
    \MF(\A^2,-x^3s) \simeq D^b(\C[x]/x^3),
\end{equation*}
where again the $R$-charge is $|x|=1$. The subcategory supported at $\{s=0\}$
corresponds to $\Perf(\C[x]/x^3)$, which shows that the kernel subcategory of
this resolution is generated by an object with graded endomorphism algebra
$\C[x]/x^3$, where $|x|=1$. This precisely matches \cite[Proposition 4.6]{F}.

\subsubsection{Type $A_3$} \label{sec:A_3}

Consider the $A_3$ curve singularity $\{y^2=x^4\}\subset\A^4$. We will
construct two categorical resolutions of this singularity using
compactifications of LG-models, and show that the two are related via a
semi-orthogonal decomposition by combining \cref{thm:main} with variation of
GIT. This is motivated by the imprecise conjecture that any ``nice enough''
categorical resolution should contain a unique minimal categorical resolution
as an admissible component.

Firstly, we construct a partial compactification with no orbifold structure
using the coordinate change $y^2-x^4 = \gamma(\gamma-\chi^2)$. Here
$\gamma=y+x^2$ and $\chi=\sqrt2x$. Using Kn\"orrer periodicity, we start with
the LG-model
\begin{equation*}
    (\A^3,(y^2-x^4)p) = (\A^3,\gamma(\gamma-\chi^2)p).
\end{equation*}
The categorical resolution is given by the partial compactification
\begin{equation*}
    ([\O\oplus\O(-1)]_{\P^1},\gamma(\gamma\sigma-\chi^2)p),
\end{equation*}
which is a phase of the GIT problem $\C^4/\C^*$ with coordinates
$(\chi,\gamma,\sigma,p)$ and weights $(0,-1,1,1)$. Variation of GIT shows that
this categorical resolution has an SOD of the form $\langle
D^b(\A^1),D^b(\A^1)\rangle$, which is related to the construction of the
singularity as a gluing of two copies of $\A^1$.\footnote{One can construct
categorical resolutions of singularities which are formed by gluing smooth
components by taking suitable semi-orthogonal gluings of the components
\cite{L}.} Unlike \cref{sec:A_2}, we do not get a semi-orthogonal decomposition
relating it directly to the normalization of the singularity. Due to the lack
of orbifold structure, we expect this to be the minimal categorical resolution.

In contrast, there is a non-minimal categorical resolution given by the
following orbifold partial compactification, without requiring a coordinate
change:
\begin{equation*}
    ([\O(-1)\oplus\O(-2)]_{\P^{[1:4]}},(y^2-x^4)p).
\end{equation*}
This is a phase of the GIT problem $\C^4/\C^*$ with coordinates $(x,y,s,p)$ and
weights $(-1,-2,1,4)$, and the other phase is equivalent to the normalization
of the curve by Kn\"orrer periodicity.\footnote{Specifically, this realizes the
normalization as a one-step $(1,2)$-weighted blowup.} Hence, this categorical
resolution can be constructed by adjoining two semi-orthogonal exceptional
objects to the normalization of the curve. These objects are the torsion sheaf
$\O/(x,y)$ and its line bundle twist $\O(1)/(x,y)$.

Call these two categorical resolutions $\scrC_{\min{}}$ and $\scrC_\norm$:
\begin{align*}
    \scrC_{\min{}}
        &\coloneqq \MF([\O\oplus\O(-1)]_{\P^1},\gamma(\gamma\sigma-\chi^2)p), \\
    \scrC_\norm
        &\coloneqq \MF([\O(-1)\oplus\O(-2)]_{\P^{[1:4]}},(y^2-x^4)p).
\end{align*}
We claim that there is a semi-orthogonal decomposition
\begin{equation} \label{eqn:sod}
    \scrC_\norm = \langle\scrC_{\min{}}, E_1, E_2\rangle,
\end{equation}
where $E_1$ and $E_2$ are exceptional objects contained in the kernel
subcategory of $\scrC_\norm$, given by the torsion sheaf $\O/(s,y-x^2)$ and its
line bundle twist $\O(1)/(s,y-x^2)$.

That $E_1$ and $E_2$ are contained in the kernel subcategory is clear, since
they are supported at $s=0$, and that they are exceptional (and
semi-orthogonal) follows from direct computation. That their orthogonal
complement is equivalent to $\scrC_{\min{}}$, however, is non-obvious. The two
LG-models have affine open covers where one open set is the original LG-model
$(\A^3,(y^2-x^4)p)$, while the other open set $\{p\ne0\}$ is different in each
model. We will justify \cref{eqn:sod} by showing that these different open sets
containing the compactification divisors are related by an analogous SOD, using
the fact that $E_1$ and $E_2$ are supported at the boundary. The global
decomposition in \cref{eqn:sod} follows by gluing this together with the
identity on $(\A^3,(y^2-x^4)p)$.

So, we are interested in the following LG-models (the two charts $\{p\ne0\}$):
\begin{equation*}
    \scrC^\infty_{\min{}}
        \coloneqq \MF(\A^3,\gamma(\gamma\sigma-\chi^2))
    \qquad \text{and} \qquad
    \scrC^\infty_\norm
        \coloneqq \MF([\A^3/\Z_4],y^2-x^4).
\end{equation*}
For $\scrC^\infty_{\min{}}$ the $R$-charges are
$(|\chi|,|\gamma|,|\sigma|)=(0,2,-2)$. For $\scrC^\infty_\norm$ the
$\Z_4$-weights are $(-1,-2,1)$ and the $R$-charges are
$(|x|,|y|,|s|)=(\frac{1}{2},1,-\frac{1}{2})$.

Applying \cref{thm:main} to the $\gamma\chi^2$-term in
$\gamma(\gamma\sigma-\chi^2)$ gives
\begin{equation*}
    \scrC^\infty_{\min{}} \simeq \MF(\A^2,\zeta^4\sigma)
        \simeq D^b(\C[\zeta]/\zeta^4),
\end{equation*}
where $\zeta=\sqrt\gamma$ has $R$-charge $|\zeta|=1$, from a double cover
branched over $\{\gamma=0\}$. The second equivalence is Kn\"orrer periodicity
for the hypersurface $\{\zeta^4=0\}\subset\A^1$, and the kernel subcategory
supported at $\{\sigma=0\}$ is equivalent to $\Perf(\C[\zeta]/\zeta^4)$.

We will study $\scrC^\infty_\norm$ using the same trick, so we first switch to
the same coordinates:
\begin{equation*}
    \scrC^\infty_\norm = \MF([\A^3/\Z_4],\gamma(\gamma-\chi^2)),
\end{equation*}
where $\chi$ and $\gamma$ are homogeneous with the same weights as $x$ and $y$.
Applying \cref{thm:main} to the $\gamma\chi^2$-term gives
\begin{equation*}
    \scrC^\infty_\norm \simeq \MF([\A^2/\Z_4],\zeta^4),
\end{equation*}
with $\Z_4$-weights $(1,1)$ and $R$-charges
$(|\zeta|,|s|)=(\frac{1}{2},-\frac{1}{2})$. This is a phase of the GIT problem
$\C^3/\C^*$ with weights $(1,1,-4)$ and $R$-charges
$(|\zeta|,|s|,|p|)=(1,0,-2)$. The other phase is
\begin{equation*}
    \MF(\O(-4)_{\P^1},\zeta^4p) \simeq D^b(\C[\zeta]/\zeta^4)
        \simeq \scrC^\infty_{\min{}},
\end{equation*}
so variation of GIT gives us our desired semi-orthogonal decomposition
\begin{equation*}
    \scrC^\infty_\norm = \langle\scrC^\infty_{\min{}},E_1,E_2\rangle,
\end{equation*}
where $E_1$ and $E_2$ come from torsion sheaves on the unstable locus
$\{\zeta=s=0\}$. Unravelling the equivalences involved leads to the explicit
description of $E_1$ and $E_2$ given initially.

In particular, the kernel subcategory of $\scrC_{\min{}}$ is generated by an
object with graded endomorphism algebra $\C[\zeta]/\zeta^4$, where $|\zeta|=1$.
The kernel subcategory of $\scrC_\norm$ has this as an admissible component,
but also contains $E_1$ and $E_2$.

\printbibliography

\end{document}